\documentclass[11pt]{article}

\usepackage{amsthm, amssymb, bm}
\usepackage{soul, color}
\usepackage[]{amsmath}
\usepackage[]{amsfonts}
\usepackage[]{fancyhdr}
\usepackage[]{graphicx}
\graphicspath{{/EPSF/}{../figures/}{figures/}}

\usepackage[colorinlistoftodos]{todonotes}

\newtheorem{theorem}{Theorem}[]
\newtheorem{lemma}[theorem]{Lemma}
\newtheorem{proposition}[theorem]{Proposition}
\newtheorem{corollary}[theorem]{Corollary}

\newtheorem{remark}[theorem]{Remark}

\newtheorem{problem}{Problem}[section]

\def \Cm {\mathbb{C}}

\def \Nm {\mathbb{N}}
\def \Rm {\mathbb{R}}
\def \Sm {\mathbb{S}}

\def \Zm {\mathbb{Z}}

\def\F{\mathcal{F}}

\def\SS{\mathcal{S}}
\def\R{\mathcal{R}}

\newcommand{\cout}[1]{}

\newcommand{\x}{\mathrm{x}}

\newcommand{\sgn}[1]{\,{\rm sign}(#1)}

\newcommand{\dbar}{\overline{\partial}}
\newcommand{\zbar}{\overline{z}}

\newcommand{\adj}{I_0^\sharp}
\newcommand{\adjmu}{I_0^*}

\newcommand{\Dm}{ \mathbb{D} }
\renewcommand{\ss}{ \mathfrak{s} }

\renewcommand{\F}[1]{{ #1}} 
\newcommand{\wtH}{\widetilde{H}}

\newcommand{\dvolk}{dVol_\kappa}

\newcommand{\DD}{ {\mathcal L}}

\title{Functional relations, sharp mapping properties and regularization of the X-ray transform on disks of constant curvature}
\author{Fran\c{c}ois Monard\thanks{Department of Mathematics, University of California, Santa Cruz CA 95064; email: fmonard@ucsc.edu.}}

\begin{document}

\maketitle

\begin{abstract} On simple geodesic disks of constant curvature, we derive new functional relations for the geodesic X-ray transform, involving a certain class of elliptic differential operators whose ellipticity degenerates normally at the boundary. We then use these relations to derive sharp mapping properties for the X-ray transform and its corresponding normal operator. Finally, we discuss the possibility of theoretically rigorous regularized inversions for the X-ray transform when defined on such manifolds.    
\end{abstract}


\section{Introduction}

Consider $(M,g)$ a \F{closed,} simple\footnote{\F{A Riemannian surface is called {\em simple} if it is non-trapping (no infinite-length geodesic), has strictly convex boundary (in the sense of the second fundamental form) and has no conjugate points, see e.g. \cite{Ilmavirta2019}.}} Riemannian surface with boundary with unit tangent bundle $SM$ and inward-pointing boundary 
\begin{align*}
    \partial_+ SM = \{ (\x,v)\in SM,\ \x\in \partial M,\ \mu(\x,v) := g_\x (v,\nu_\x) >0\}, 
\end{align*}
with $\nu_\x$ the inward-pointing unit normal. We equip $M$ with its Riemannian area $dVol_g$ and $\partial_+ SM$ with the area form $d\Sigma^2$, product of the arclength measure on $\partial M$ and the Euclidean measure on the tangent circles. Our object of study is the geodesic X-ray transform $I_0\colon L^2(M) \to L^2(\partial_+ SM, d\Sigma^2)$, defined by 
\begin{align}
    I_0 f(\x,v) := \int_0^{\tau(\x,v)} f(\gamma_{\x,v}(t))\ dt, \qquad (\x,v)\in \partial_+ SM,
    \label{eq:Xray}
\end{align}
where $\gamma_{\x,v}(t)$ denotes the unit-speed geodesic with initial condition $(\gamma(0), \dot \gamma(0)) = (\x,v)$, whose first exit time out of $M$ is $\tau(\x,v)$. Denote $\adjmu$ the adjoint of $I_0$ in this setting. The function $I_0 f$ is really a function on the space of geodesics represented by $\partial_+ SM$ if $M$ is convex, non-trapping. \F{This parameterization is referred to as 'fan-beam' coordinates in the Euclidean literature applied to X-ray CT \cite{Natterer2001}.} Generally on simple surfaces, such a transform is known to be injective \cite{Mukhometov1975}, and Fredholm type inversion formulas were derived in \cite{Pestov2004,Krishnan2010} and implemented in \cite{Monard2013}, with a compact error term shown to vanish in cases with constant curvature, see also the recent topical review \cite{Ilmavirta2019}. \F{The X-ray transform can also be explicitly inverted in certain non-compact, symmetric spaces, see e.g. \cite{Helgason2010,Natterer2001,Bal2005}. The present article focuses on the compact case, where every geodesic reaches the boundary in finite time, addressing sharp mapping properties for the geodesic X-ray transform, that take into account boundary behavior. The results currently hold in the case where $M$ is a geodesic disk of constant curvature, that is to say, obtained as the image of a centered ball (on the tangent space of some point on a surface of constant curvature) by the Riemannian exponential map.}

The first issue to be discussed is the choice of co-domain topology for $I_0$. Indeed, the transform \eqref{eq:Xray} is often studied in the functional setting $L^2(M)\to L^2(\partial_+ SM, \mu d\Sigma^2)$ for which the adjoint will be denoted by $\adj$. In that case, the operator $\adj I_0$ is a classical $\Psi DO$ which can be naturally extended to a simple open neighborhood of $M$ and satisfies a $\frac{-1}{2}$-transmission condition at $\partial M$. In particular, the Boutet de Monvel calculus and its generalization have been used in \cite{Monard2017} to obtain mapping properties of $\adj I_0$. In the present case, where $\adjmu I_0 = \adj \frac{1}{\mu} I_0$, the operator so obtained no longer extends, and thus transmission conditions are not available. On the other hand, $L^2(\partial_+ SM, d\Sigma^2)$ is precisely the co-domain topology where Singular Value Decompositions for \eqref{eq:Xray} are known in some cases \cite{Maass1992,Mishra2019}.  

\paragraph{New functional relations.} The first salient feature of this article is the derivation of new functional relations between the normal operator $I_0^* I_0$ and a distinguished second-order differential operator. In non-compact spaces, the relation $R^t R = (-\Delta)^{-1/2}$ can be derived for the Radon transform $R$ on the Euclidean plane ($R^t$ denotes the transpose). More generally, examples of Radon transforms on two-point homogeneous spaces abound, where the corresponding normal operator can be inverted using some differential operator \cite{Helgason2010}; see also \cite{Guillarmou2017}, where an explicit relation between $I_0^* I_0$ and the Laplace-Beltrami operator on hyperbolic surfaces of constant curvature was derived. 

When restricting these transforms to compact domains, using global functional relations may not translate into relations on compact domains, see Remark \ref{rmk:different}. Thus this article presents a link between X-ray transforms on simple surfaces and second-order elliptic operators whose ellipticity degenerates non-tangentially at a specific order at the boundary. Such operators also appear under the name of Kimura type operators studied in the works \cite{Epstein2013} for their applications to population genetics, and analyzed through the lens of the calculus of uniformly degenerate (or 0-) operators \cite{Mazzeo1991}. In the case of the \F{closed} Euclidean unit disk \F{$\Dm = \{(x,y)\in \Rm^2,\ x^2+y^2\le 1\}$}, this relation becomes
\begin{align}
    \DD (I_0^* I_0)^2 = (I_0^* I_0)^2 \DD = 4\pi\ \text{id}|_{C^\infty (\Dm)},
    \label{eq:sqrt}
\end{align}
where $\DD$ is given by, in polar coordinates $(\rho,\omega)$
\begin{align}
    \DD = - \left( (1-\rho^2) \partial_\rho^2 + \left( \frac{1}{\rho} - 3\rho \right) \partial_\rho + \frac{1}{\rho^2} \partial_\omega^2 \right) + 1.
    \label{eq:DD}
\end{align}
Although this specific case can be pieced together using existing literature on Zernike polynomials \cite{Zernike1934} and the X-ray transform \cite{Cormack1964,Kazantsev2004}, to the author's knowledge, such a result was not \F{explicitly} stated in the literature. Moreover, we show in this article that this link between X-ray transforms on simple surfaces and degenerate elliptic operators persists when $M$ is a geodesic disk of arbitrary radius in constant curvature spaces, see in particular Theorem \ref{thm:main1}. In fact, more is at play: there exists a smooth, non-vanishing weight $w$ such that the operator $I_0 w := I_0 (w\cdot)$ intertwines an operator similar to $\DD$ above, with an operator $-T^2$, where $T$ is some vector field on $\partial_+ SM$, i.e. 
\begin{align}
    I_0 w \circ \DD = (-T^2) \circ I_0 w.
    \label{eq:egorov}
\end{align}
On the Euclidean unit disk, $w\equiv 1$ and $T = \partial_\beta - \partial_\alpha$ in fan-beam coordinates. Similar intertwining properties have also been very useful to the analysis of generalized Radon transforms on symmetric spaces \cite{Gonzalez2006,Helgason2010,Kakehi1999} and one-dimensional convolution problems \cite{Grunbaum1983,Maass1991}, some of which arise naturally from integral geometric problems. 

\paragraph{Mapping properties.} Identities \eqref{eq:sqrt} and \eqref{eq:egorov} bring us to the second topic of interest of this article, namely the mapping properties of $I_0$ and $\adjmu I_0$. Recently, range characterization and obtaining sharp mapping properties for X-ray transforms have regained interest \cite{Kumar2010,Rubin2013,Sharafutdinov2016,Assylbekov2018,Boman2018,Mishra2019}, with the challenges of accurately taking boundary behavior into account, and finding spaces on $\partial_+ SM$ which only require regularity along some but not all directions \F{(unlike usual Sobolev regularity which requires controlling derivatives along all directions)}. This is because, as recently pointed in \cite{Assylbekov2018}, \F{although $\partial_+ SM$ is $(2d-2)$-dimensional when $M$ has dimension $d$, only $d-1$ vector fields on $\partial_+ SM$} are needed to be fully elliptic on the image of the canonical relation of $I_0$ when viewed as an FIO. The typical example of this is in parallel Euclidean geometry, where regularity with respect to $\frac{d}{ds}$ is sufficient \cite{Natterer2001}. Recently in \cite{Assylbekov2018}, a construction of Sobolev spaces based on \F{extending} $M$ and encoding smoothness with respect to a reduced number of vector fields was indeed possible in order to capture the smoothing properties of the X-ray transform. In the recent work \cite{Paternain2019}, other spaces involving regularity with respect to tangential-horizontal directions on $\partial_+ SM$ are defined, allowing the authors to formulate sharp $L^2-H^{1/2}$ stability estimates on manifolds of non-positive curvature for the X-ray transform defined on tensor fields. 

To approach this question here, the functional relations \eqref{eq:sqrt} and \eqref{eq:egorov} suggest two things: relation \eqref{eq:sqrt} suggest that the mapping properties of $\adjmu I_0$ are best described on a Sobolev scale where smoothness is encoded with respect to $\DD$; relation \eqref{eq:egorov} suggests that on the side of $\partial_+ SM$, smoothness w.r.t. $\DD$ will be translated into smoothness w.r.t. $T$. Such statements are made precise in Section \ref{sec:mapping}, where appropriate Sobolev scales are introduced, and where sharp mapping properties for $\adjmu I_0$ and $I_0$ are formulated for any order order on that scale, see Corollary \ref{cor:1} and Theorem \ref{thm:main2}. The Hilbert scale introduced on $\Dm$, denoted $\wtH^s(\Dm)$ below, has both a definition in terms of powers of $\DD$, and in terms of decay rate of Zernike \F{polynomial} expansions. A similar though inequivalent scale using different weighting was also defined in \cite[Eq. (2.9)]{Johnstone1990} to describe {\it ad hoc} smoothness classes there.

Strikingly, while $\adj I_0 (L^2(M)) \subset H^1(M)$ as proved in \cite{Monard2017}, we now have $\adjmu I_0 (L^2(M)) \supsetneq H^1(M)$. For higher-order Sobolev spaces, the mapping properties of $\adj I_0$ require H\"ormander type transmission spaces which require microlocal tools in order to be defined, whereas the present definitions are rather transparent. 

\paragraph{Regularization.} As a consequence of the previous derivations, we finally discuss a new approach to regularization of geodesic X-ray transforms. As the transform \eqref{eq:Xray} is smoothing of order $\frac{1}{2}$, its stable inversion requires regularization, a theory only rigorously developed in the Euclidean case, and in parallel geometry. There, with the help of the Fourier Slice Theorem, one may derive filtered-backprojection type formulas \cite[Theorem 1.3]{Natterer2001} (with, e.g., filter $h$)
\begin{align}
    R^t h \star f = R^t (h\star Rf), 
    \label{eq:FBP}
\end{align}
where the left convolution is two-dimensional and the right one is one-dimensional. As $h$ is typically a smoothing kernel (in the Fourier domain, a low-frequency version of $|\sigma|$), these formulas give a theoretically exact estimation of how the reconstruction $f$ is smoothed out by the kernel $R^t h$, upon processing the data $Rf$ in a practically efficient way (the column-wise convolution by $h$ can be carried out by Fast Fourier Transform, and the backprojection $R^t$ is unavoidable). Unfortunately such formulas do not exist in a curved setting, let alone in fan-beam coordinates on the Euclidean disk. The last aim of this article is to present a new approach to tackle this issue, which is theoretically exact on the class of surfaces considered, see Section \ref{sec:regularization}. Implementation of such formulas will appear in future work.

\section{Main results}

\subsection{New functional relations}

We will work with simple geodesic disks in constant curvature spaces, modeled over the two-parameter family of domains $\Dm_R = \{(x,y)\in \Rm^2,\ x^2+y^2\le R^2\}$, endowed with the metric $g_{\kappa}(z) = (1+\kappa |z|^2)^{-2} |dz|^2$. Such models $(\Dm_R, g_\kappa)$ have constant curvature $4\kappa$ and are simple if and only if $R^2 |\kappa|<1$ (the case $R^2 \kappa = 1$ gives a hemisphere, of totally geodesic boundary; the case $R^2\kappa = -1$ gives a Poincar\'e hyperbolic model, non-compact). The first result of this article is as follows. 

\begin{theorem}\label{thm:main1}
    Let $(M,g)$ a simple geodesic disk of constant curvature, modeled on $(\Dm_R, g_\kappa)$ for some $(\kappa,R)$ satisfying $R^2 |\kappa|<1$, and consider the geodesic X-ray transform $I_0$ defined in \eqref{eq:Xray}, with adjoint $\adjmu$. Then there exists a second-order differential operator $\DD$ on $M$, a first-order differential operator $T$ on $\partial_+ SM$, and a non-vanishing weight function $w\in C^\infty(M)$ such that 
    \begin{align}
	\DD\circ\adjmu = \adjmu\circ (-T^2), \qquad I_0 w \circ \DD = (-T^2)\circ I_0 w. 
	\label{eq:rel1}
    \end{align}
    The operator $\DD$ is a degenerate elliptic differential operator of Kimura type, positive, coercive and formally self-adjoint on $L^2(M,w\ dVol_g)$. Moreover, the following relation holds
    \begin{align}
	\DD (\adjmu I_0 w)^2 = c^2_{\kappa,R}\ id|_{C^\infty(M)}, \qquad c_{\kappa,R} := \frac{4\pi R}{1-\kappa R^2}
	\label{eq:rel2}
    \end{align}
\end{theorem}
In the statement above, by 'Kimura type' we mean that the operator $\DD$ is elliptic at interior points, and if $r$ is a boundary defining function for $\Dm_R$ and $\omega$ is the polar variable, the operator $\DD$ is of the form $A r \partial_r^2 + B \partial_\omega^2$ (with $A>0$ and $B>0$) up to lower-order terms near the boundary\footnote{$A,B$ can in general be functions of $\omega$, see \cite{Epstein2013} for general definitions.}. 

In \eqref{eq:rel2}, that the natural space is $C^\infty(M)$ is in fact proved slightly later, in Lemma \ref{lem:intersect} below. The operators $\DD$ and $T$ are defined in \eqref{eq:LTkappa}, in terms of the reference case \eqref{eq:DD} and \eqref{eq:Te}, and $(R,\kappa)$-dependent intertwining diffeomorphisms $\Phi$ defined in \eqref{eq:Psi} and $\ss$ defined in \eqref{eq:skappa}.

\begin{remark}\label{rmk:different}
    Note that \eqref{eq:rel2} is a genuinely different scenario even in the Euclidean case, which could not be obtained from using the classical relation $\R^t \R = (-\Delta)^{-1/2}$ and considering a restriction $r_M \R^t \R e_M$ where $e_M,r_M$ are operators of extension-by-zero and restriction. Indeed in the latter case, the following isomorphism property
    \begin{align*}
	r_M \R^t \R e_M \colon d_M^{-1/2} C^\infty(M) \leftrightarrow C^\infty(M),
    \end{align*}
    is a special case of \cite[Theorem 4.4]{Monard2017}, with $d_M$ a boundary defining function for $M$. Though the operator is $L^2(M)-L^2(M)$ self-adjoint, these smooth mapping properties make it difficult to envision a relation of the kind \eqref{eq:rel2}. 
\end{remark}

\subsection{Range characterization and mapping properties} \label{sec:mapping}

Relations \eqref{eq:rel1}-\eqref{eq:rel2} indicate that, upon defining appropriate Hilbert scales modeled after $\DD$ and $T$, one may formulate accurate mapping properties for $I_0$. Specifically, one may naturally define two Sobolev type of scales of spaces indexed over $s\in \Rm$. The first one on $M$ is given by
\begin{align}
    \wtH^s (M) := \left\{ f\in L^2(M, w\ dVol_g),\quad \DD^{s/2} f \in L^2(M, w\ dVol_g) \right\}, 
    \label{eq:wtH}
\end{align}
and an important property is the following:
\begin{lemma}\label{lem:intersect} Let $(M,g)$ modeled on $(\Dm_R, g_\kappa)$ with $R^2|\kappa|<1$. Let $\DD$ as in Theorem \ref{thm:main1} and the Hilbert scale $\{\wtH^s(M)\}$ as in \eqref{eq:wtH}. Then 
    \begin{align*}
	\cap_{s\ge 0} \wtH^s(M) = C^\infty(M).
    \end{align*}    
\end{lemma}

Then as an immediate consequence of Theorem \ref{thm:main1}, the following mapping properties are immediate
\begin{corollary}\label{cor:1}
    Let $M, g, \kappa, R, \DD, T, w$ defined as in Theorem \ref{thm:main1}. Then,
    \begin{align*}
	\adjmu I_0 w (\wtH^s(M)) &= \wtH^{s+1}(M), \qquad \forall s\in \Rm, \\
	\adjmu I_0 (C^\infty(M)) &= C^\infty(M).
    \end{align*}
\end{corollary}

\begin{remark}
    Corollary \ref{cor:1} is in stark contrast with the isomorphism properties
    \begin{align*}
	\adj I_0 &\colon H^{-1/2,(s)}(M) \longleftrightarrow H^{s+1}(M), \qquad s>-1, \\
	\adj I_0 &\colon d_M^{-1/2} C^\infty(M) \longleftrightarrow C^\infty(M),
    \end{align*}
    proved in \cite{Monard2017} for any simple Riemannian surface $(M,g)$. Above, $d_M$ is a smooth function on $M$ equal to dist$(\x,\partial M)$ in a neighborhood of $\partial M$, and $H^{-1/2,(s)}(M)$ denotes a scale of H\"ormander $(-1/2)$-transmission spaces, whose intersection is $d_M^{-1/2} C^\infty(M)$. These differences show in particular the crucial role played by the weight $\frac{1}{\mu}$. 
\end{remark}

To obtain mapping properties of $I_0$, it is not enough to fully understand the smoothing properties of $I_0$, but one must also account for the infinite-dimensional cokernel of this operator. 

On the $\partial_+ SM$ side, we first define the relation $\SS_A$ which is the composition of the scattering relation and the antipodal map. An important space in our analysis will be 
\begin{align}
    C_{\alpha,-,+}^\infty(\partial_+ SM) := \{ u\in C^\infty(\partial_+ SM),\ A_- u \text{ is smooth and fiberwise odd on } \partial_+ SM \},
    \label{eq:Calmp}
\end{align}
see also \cite[Appendix A]{Mishra2019}, where $A_-$ turns a function on $\partial_+ SM$ into its odd extension to $\partial SM$ with respect to the scattering relation. In our circularly symmetric cases, the space $L^2(\partial_+ SM, d\Sigma^2)$ splits orthogonally into $L^2_+ \oplus L^2_-$, where
\begin{align*}
    L^2_{\pm}(\partial_+ SM, d\Sigma^2) = L^2(\partial_+ SM, d\Sigma^2) \cap \ker (id \mp \SS_A^*).
\end{align*}
In our case, the action takes place in $L^2_+$ since $I_0 f$ does not depend on the orientation of a geodesic. One may then show that with $T$ defined in Theorem \ref{thm:main1}, and upon looking at smooth elements, $T(\ker (id\pm \SS_A^*)) \subset \ker (id\mp \SS_A^*)$, in particular, $\ker (id \pm \SS_A^*)$ is stable under $-T^2$. This justifies the construction of the following Hilbert scale:  
\begin{align}
    H_{T,+}^s(\partial_+ SM) := \left\{u\in L_+^2(\partial_+ SM, d\Sigma^2), \quad (-T^2)^{s/2} u \in L_+^2 (\partial_+ SM, d\Sigma^2) \right\},
    \label{eq:HT}
\end{align}
whose intersection can be shown to be nothing but
\begin{align}
    \cap_{s\ge 0} H_{T,+}^s (\partial_+ SM) = C_{\alpha,-,+}^\infty(\partial_+ SM).
    \label{eq:intersectHT}
\end{align}

Such a cokernel can in fact be fully described as the $L^2$-orthocomplement of the kernel of a natural operator 
\begin{align}
    C_- \colon L_+^2(\partial_+ SM, d\Sigma^2) \to L_+^2(\partial_+ SM, d\Sigma^2), \qquad C_- := \frac{1}{2} A_-^* H_- A_-,
    \label{eq:Cminus}
\end{align}
where $A_-$ denotes antisymmetrization with respect to scattering relation, $H_-$ is the odd fiberwise Hilbert transform on the fibers of $\partial SM$, and $A_-^*$ denotes the $L^2-L^2$ adjoint of $A_-$. In all cases considered, $C_-$ commutes with $-T^2$, and this implies that $C_-\colon H_{T,+}^s\to H_{T,+}^s$ is well-defined for all $s\in \Rm$.

With these definitions, we can now formulate our second main result: 

\begin{theorem}\label{thm:main2}
    Let $(M, g)$ modeled on $(\Dm_R, g_\kappa)$ with $R^2 |\kappa|<1$, let $\DD, T, w$ as in Theorem \ref{thm:main1}, and let $C_-$ defined in \eqref{eq:Cminus}. Then,
    \begin{align*}
	I_0 (\wtH^s(M)) &= \left\{ w\in H_{T,+}^{s+\frac{1}{2}} (\partial_{+} SM), \qquad C_- w = 0 \right\}, \qquad \forall s\in \Rm, \\
	I_0 (C^\infty(M)) &= \left\{ w\in C_{\alpha,-,+}^\infty(\partial_+ SM), \qquad C_- w = 0 \right\}.
    \end{align*}
    Moreover, we have the following equality, for all $s\in \Rm$ and $f\in \wtH^s(\Dm_R)$
    \begin{align}
	\|f\|_{\wtH^s(\Dm_R)} = \frac{1}{\sqrt{c_{\kappa,R}}} \|I_0 (wf)\|_{H^{s+1/2}_{T,+}(\partial_+ S_{(\kappa)}\Dm_R)}.
	\label{eq:sharp}
    \end{align}
\end{theorem}
The last equality \eqref{eq:sharp} provides both a continuity estimate and a stability estimate, with explicit control of the constants. As explained in \F{Proposition} \ref{rmk:SZ}, the scale $\wtH^s(\Dm)$ is inequivalent to the classical Sobolev scale on $\Dm$. As explained in \F{Proposition} \ref{rmk:otherscales}, despite the scale $H_{T,+}^s(\partial_+ SM)$ also being inequivalent to the classical Sobolev scale on $\partial_+ SM$, it is equivalent to it on the range of $I_0$. In particular, \F{if one defines as in \eqref{eq:classic} below a more ``classical'' Sobolev scale $H_+^{s}(\partial_+ SM)$ based on regularity control over {\em all} directions,} one can formulate continuity estimates for $I_0 \colon \wtH^s(\Dm) \to H^{s+1/2}_{+}(\partial_+ SM)$. Though these \F{results} are formulated for the Euclidean case only, they are expected to carry over straightforwardly to the other cases covered here. 

The proofs of Theorem \ref{thm:main1}, Lemma \ref{lem:intersect} and Theorem \ref{thm:main2} all rely on explicit calculations. The case of the Euclidean disk is worked out as reference case, then the general model $(\Dm_R, g_\kappa)$ is deduced from results on the reference case, through the use of intertwining diffeomorphisms. The latter intertwiners are reminiscent of the definition of the {\em factorization property} appearing in \cite{Palamodov2000}, and the recent results on pairs of generalized Funk transforms \cite{Agranovsky2019a,Agranovsky2019}, which heavily rely on intertwining diffeomorphisms relating the classical Funk transform with generalizations where integration is done along slices of the sphere by hyperplanes passing through a fixed point that is different from the origin, see also \cite{Beltukov2009}.

The relation \eqref{eq:rel2} (and the mere existence of such an $\DD$) may not be expected to hold for general surfaces, as the circular symmetry and the constancy of the curvature both seem to play important roles here. In addition, it is hopeless to expect a relation of the form \eqref{eq:rel1} in the presence of conjugate points, as the works \cite{Stefanov2012a,Monard2013b,Holman2015} show that the singular support of $\adj I_0$ in the interior of $M$ (hence of $\adjmu I_0$) contains strictly more than the conormal bundle to the diagonal of $M\times M$, and as such could not possibly be inverted, even microlocally, using a $\Psi$DO. However, it is fair to ask the following question: 

\begin{problem}
    Find a characterization of all simple surfaces-with-boundary $(M,g)$ where one can prove Theorems \ref{thm:main1} and \ref{thm:main2}.
\end{problem}

\subsection{Inversion formulas and regularization theory}
\label{sec:regularization}

\F{We end with a self-contained description of an important consequence of the results above, for the purpose of a theorically rigorous approach to regularized inversion in fan-beam coordinates. }

The operator $\DD$ appearing in Theorem \ref{thm:main1} is always a self-adjoint, unbounded operator on the space $L^2(M, w\ dVol_g)$, thus by the spectral theorem for unbounded self-adjoint operators, one can make sense of $F(\DD)$ for a large class of functions $F$ containing real powers. In addition, relations \eqref{eq:rel1} imply that for every such $F$,
\begin{align}
    I_0 w \circ F(\DD) = F(-T^2) \circ I_0 w, \qquad \F{F(\DD)\circ I_0^* = I_0^* \circ F(-T^2).}
    \label{eq:intertwF}
\end{align}
With $F (s) = s^{\alpha}$ and using \eqref{eq:rel1}-\eqref{eq:rel2}, this provides a family of new inversion formulas\F{.}
\begin{theorem} \label{thm:inversion}
    Let $(M,g)$ modeled on $(\Dm_R,g_\kappa)$ with $R^2 |\kappa|<1$, and let $\DD, T$ \F{and $c_{\kappa,R}$} as in Theorem \ref{thm:main1}. Then for all $f\in L^2(M, w\ dVol_g)$
    \begin{align}
	f = \frac{w}{c_{\kappa,R}} \DD^{\frac{1}{2}-\alpha} \circ \adjmu \circ (-T^2)^\alpha \circ I_0 f. 	
	\label{eq:inversion}
    \end{align}
\end{theorem}
\F{ \begin{proof}
	Since the proof of \eqref{eq:rel2} is seen at the spectral level, we also have $\DD^{1/2} I_0^* I_0 w = c_{\kappa,R}\ \text{id}|_{C^\infty(M)}$. We then deduce
	\begin{align*}
	    c_{\kappa,R}\ \text{id}|_{C^\infty(M)} = \DD^{1/2}\circ I_0^*\circ I_0 w = \DD^{1/2-\alpha} \circ \DD^\alpha\circ I_0^* \circ I_0 w \stackrel{\eqref{eq:intertwF}}{=} \DD^{\frac{1}{2}-\alpha}\circ I_0^*\circ (-T^2)^\alpha \circ I_0 w. 
	\end{align*}
	The result follows by left-multiplying by $w$, and right-multiplying by $w^{-1}$.
    \end{proof}
}

This family of formulas is in the spirit of \cite[\S II.2, Theorem 2.1]{Natterer2001}, where $(-T^2)^\alpha$ can be thought of as a Riesz potential in data space, and $\DD^\alpha$ can be thought of as a Riesz potential on $M$. For $\alpha = \frac{1}{2}$, equation \eqref{eq:inversion} becomes an inversion formula, to be contrasted with the Pestov-Uhlmann formula \cite{Pestov2004,Monard2015}
\begin{align}
    f = \frac{1}{4\pi} I_\perp^{\sharp}  \left(\frac{1}{2} A_+^* H_- A_-\right) I_0 f,
    \label{eq:PU}
\end{align}
where the backprojection operator $I_\perp^{\sharp}$ contains the differentiation step and is not the direct adjoint of $I_0$. On the other hand, the main challenge of the current formula is to find explicit ways to compute $(-T^2)^{1/2} = |T|$. Even more generally, one may be interested in regularizing \eqref{eq:inversion} or \eqref{eq:PU}, since $I_0$ is smoothing of order $1/2$ and its inversion is a mildly unstable process, sensitive to noise. To this end, the relation \eqref{eq:intertwF} gives the possibility of theoretically exact regularized reconstruction formulas, by combining all three \F{equations} above and assuming that $F$ is a low-pass filter in the sense that $\lim_{s\to \infty} F(s) = 0$. The strength of the regularization depends on the rate of decay of $F$ at infinity. 

\begin{theorem} \label{thm:regularization}
    Let $F$ \F{be} a low-pass filter, then the regularized reconstruction formulas hold: for all $f\in L^2(M, w\ dVol_g)$,
    \begin{align*}
	w F(\DD) \frac{1}{w} f &= \frac{w}{c_{\kappa,R}} \adjmu \circ (-T^2)^{1/2} F(-T^2) \circ I_0 f \\
	&= \frac{1}{4\pi} I_\perp^{\sharp} \circ  \left(\frac{1}{2} A_+^* H_- A_-\right) F(-T^2) \circ I_0 f.
    \end{align*}
\end{theorem}

While these formulas hold for general filters, the choice of appropriate filters is guided by various practical reasons (e.g., methods of implementation, avoiding 'ringing' effects). A discussion on these filters and appropriate methods of implementation is reserved for future work. 

\medskip

{\bf Outline.} The remainder of the article is organized as follows. We first cover proofs of Theorem \ref{thm:main1} and Lemma \ref{lem:intersect}, by first covering the Euclidean unit disk in Section \ref{sec:Euclidean}, followed by simple geodesic disks of constant curvature in Section \ref{sec:CCD}. We then prove Theorem \ref{thm:main2} in Section \ref{sec:mappingI0}. Some facts about Zernike polynomials and proofs of auxiliary lemmas are relegated to Appendix \ref{sec:appendix}.

\section{The Euclidean unit disk} \label{sec:Euclidean}

In the case of the Euclidean unit disk, the inward pointing boundary $\partial_+ S\Dm$ is parameterized by $(\beta,\alpha) \in \Sm^1\times (-\pi/2,\pi/2)$, where $\beta$ \F{parameterizes} the boundary point $\x = e^{i\beta}$ and $\alpha$ \F{is defined via the implicit relation $v = \binom{\cos(\beta+\pi+\alpha)}{\sin(\beta+\pi+\alpha)}$, if $v$ is a unit tangent vector} above $\x$. Then the measure on $\partial_+ SM$ is $d\Sigma^2 = d\beta d\alpha$ and in particular, we have $\mu = \cos\alpha$.  

\subsection{Intertwiners}

Let us define the operator $\adj\colon C_\alpha^\infty (\partial_+ S\Dm) \to C^\infty (\Dm)$ defined in the introduction as the formal adjoint of $I_0:L^2(\Dm) \to L^2(\partial_+ S\Dm, \mu\ d\Sigma^2)$, as well as $I_0^* := I_0^\sharp (\frac{1}{\mu}\cdot)$. Such an operator takes the form 
\begin{align}
    \adj g (\x) = \int_{\Sm^1} g(\beta_-(\x,\theta), \alpha_-(\x,\theta))\ d\theta,
    \label{eq:adj}
\end{align} 
where $\beta_-(\x,\theta)$, $\alpha_-(\x,\theta)$ are the fan-beam coordinates of the unique line passing through $(\x,\theta)$. In what follows, we will identify $\x$ with $\rho e^{i\omega}$. \F{See Figure \ref{fig:1} for a summary.

\begin{figure}[htpb]
    \centering
    \includegraphics[height=0.3\textheight]{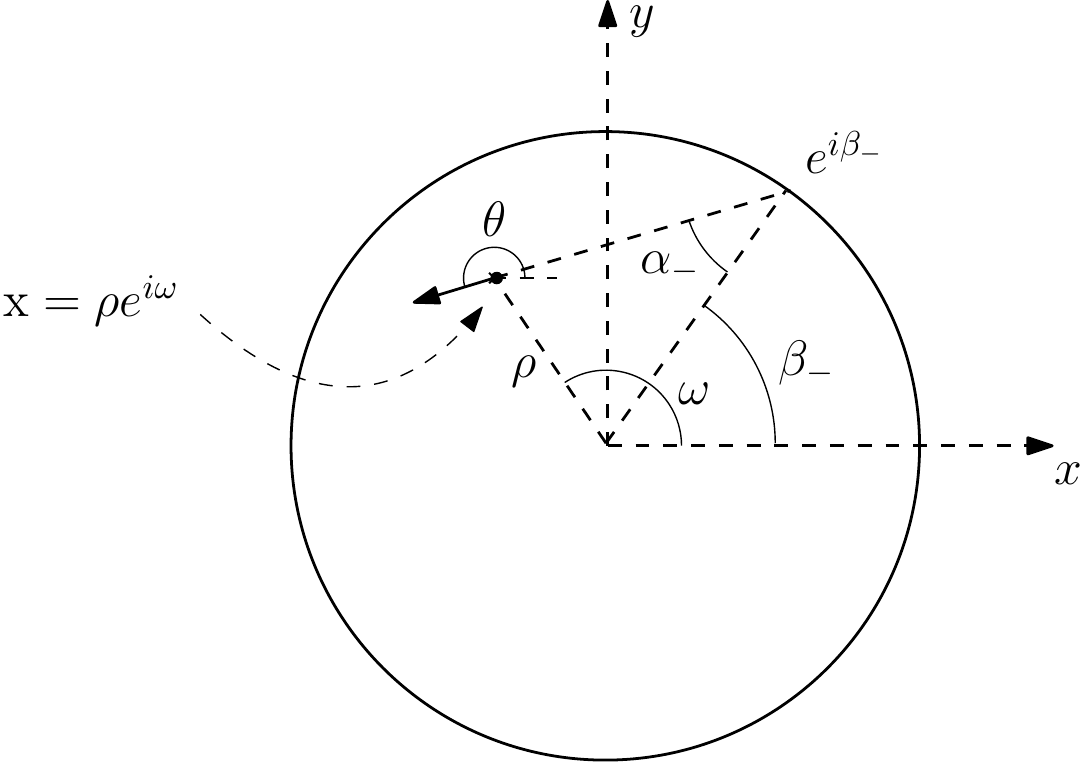}
    \caption{Setting of definition of $(\beta_-(\rho e^{i\omega}, \theta),\alpha_- (\rho e^{i\omega}, \theta))$ (written as $(\beta_-, \alpha_-)$ on the diagram). The rotation invariance implies that if $\theta$ and $\omega$ are translated by $\delta$, then $\beta_-$ is translated by $\delta$ and $\alpha_-$ remains unchanged.}
    \label{fig:1}
\end{figure}

From the observation made in Fig. \ref{fig:1},} these functions satisfy the following relation: 
\begin{align*}
    \beta_-(\rho e^{i\omega}, \theta) = \omega + \beta_- (\rho,\theta-\omega), \qquad \alpha_-(\rho e^{i\omega}, \theta) = \alpha_- (\rho,\theta-\omega).
\end{align*}
In particular, the expression of $\adj g$ immediately becomes
\begin{align*}
    \adj g (\rho e^{i\omega}) \F{ =  \int_{\Sm^1} g(\omega + \beta_-(\rho,\theta-\omega), \alpha_-(\rho,\theta-\omega))\ d\theta} =  \int_{\Sm^1} g(\omega + \beta_-(\rho,\theta), \alpha_-(\rho,\theta))\ d\theta.
\end{align*}
We then immediately see the first intertwining property
\begin{align*}
    \partial_\omega \circ \adj = \adj \circ \partial_\beta, \qquad \partial_\omega \circ \adjmu = \adjmu \circ \partial_\beta.
\end{align*}
Upon defining 
\begin{align}
    T := \partial_\beta - \partial_\alpha,
    \label{eq:Te}
\end{align}
a second intertwining property is then given as follows.

\begin{theorem}\label{thm:intertw} 
    Define the operators 
    \begin{align}
	L := (1-\rho^2) \frac{\partial^2}{\partial\rho^2} + \left(\frac{1}{\rho} - 3\rho\right) \frac{\partial}{\partial\rho} + \frac{1}{\rho^2} \frac{\partial^2}{\partial\omega^2},
	\label{eq:L}
    \end{align}
    and $D:= T^2+2\tan\alpha T$. Then we have the following intertwining properties: 
    \begin{align}
	L \circ \adj &= \adj \circ D, \label{eq:inter1} \\
	\DD \circ \adjmu &= \adjmu \circ (-T^2), \qquad \DD := -L+1. \label{eq:inter2}
    \end{align}    
\end{theorem}

\begin{proof} Proof of \eqref{eq:inter1}. In what follows, $\alpha_-$ and $\beta_-$ will be short for $\alpha_- (\rho,\theta)$ and $\beta_- (\rho,\theta)$. Note the easy two properties
    \begin{align*}
	\beta_- + \alpha_- + \pi = \theta, \qquad \sin \alpha_- = - \rho \sin \theta.
    \end{align*} 
    In particular, this gives $\frac{\partial\alpha_-}{\partial\rho} = - \frac{\sin\theta}{\cos \alpha_-} = \frac{1}{\rho} \tan \alpha_-$, $\frac{\partial \alpha_-}{ \partial \theta} = \frac{-\rho\cos\theta}{\cos\alpha_-}$, and the derivatives of $\beta_-$ can be deduced through the relations
    \begin{align*}
	\frac{\partial \beta_-}{ \partial \rho} = - \frac{\partial \alpha_-}{ \partial \rho}, \qquad \frac{\partial \beta_-}{ \partial \theta} = 1 - \frac{\partial \alpha_-}{ \partial \theta}.
    \end{align*}
    From these relations, we immediately deduce the property that 
    \begin{align*}
	\frac{\partial}{\partial \rho} \adj g = - \frac{1}{\rho} \adj [\tan \alpha Tg].
    \end{align*}
    Iterating this formula, we obtain
    \begin{align*}
	\frac{\partial^2}{\partial \rho^2} \adj g = \frac{1}{\rho^2} \adj [\tan \alpha Tg] + \frac{1}{\rho^2} \adj [\tan\alpha T (\tan\alpha T g)] = \frac{1}{\rho^2} \adj [\tan^2\alpha T^2 g - \tan^3 \alpha Tg].
    \end{align*}
    Then by direct algebra, using the last two identities, we obtain
    \begin{align}
	\begin{split}
	    [(1-\rho^2) \partial_\rho^2 + (\frac{1}{\rho}-3\rho)\partial_\rho] \adj g &= \frac{1}{\rho^2} \adj [\tan^2\alpha T^2 g - \tan\alpha(1+\tan^2\alpha)Tg] \dots \\
	    &\qquad - \adj [\tan^2 \alpha T^2g - \tan\alpha (\tan^2\alpha + 3) Tg].	    
	\end{split}
	\label{eq:ttmp}	
    \end{align}
    To obtain further identities, we write
    \begin{align*}
	0 &= \int_{\Sm^1} \partial_\theta ( g(\omega + \beta_-, \alpha_-))\ d\theta \\
	&=  \int_{\Sm^1} \left( \frac{\partial\beta_-}{\partial\theta} \partial_\beta + \frac{\partial\alpha_-}{\partial \theta} \partial_\alpha \right) g(\omega \F{+} \beta_-,\alpha_-)\ d\theta \\
	&= \adj [\partial_\beta g] + \rho \int_{\Sm^1} \frac{\cos\theta}{\cos\alpha_-} Tg (\omega + \beta_-, \alpha_-)\ d\theta,
    \end{align*} 
    as well as 
    \begin{align*}
	0 &= \int_{\Sm^1} \partial^2_\theta  ( g(\omega + \beta_-, \alpha_-))\ d\theta \\
	&= \int_{\Sm^1} \partial_\theta \left( \partial_\beta g + \frac{\rho\cos\theta}{\cos\alpha_-} Tg \right)\ d\theta \\
	&= \int_{\Sm^1} \left( \partial_{\beta}^2 g + \frac{2\rho\cos\theta}{\cos\alpha_-} T\partial_\beta g - \left( \frac{\rho\sin\theta}{\cos\alpha_-} + \rho^2 \cos^2\theta \frac{\sin\alpha_-}{\cos^3\alpha_-} \right) Tg + \frac{\rho^2\cos^2\theta}{\cos^2\alpha_-} T^2 g \right)\ d\theta. 
    \end{align*}
    From the previous identity and the fact that $T\partial_\beta = \partial_\beta T$, the second term equals $-2\adj [\partial_\beta^2 g]$. In the remaining terms, we use that $-\rho\sin\theta = \sin\alpha_-$ and $\rho^2 \cos^2 \theta = \rho^2(1-\sin^2\theta) = \rho^2 - \sin^2 \alpha_-$ and the previous equality becomes
    \begin{align*}
	\frac{1}{\rho^2} \adj [\tan^2 \alpha T^2 g - \tan\alpha(1 +\tan^2 \alpha) Tg] &= - \frac{1}{\rho^2} \adj [\partial_\beta^2 g] + \adj [-\tan\alpha(1+\tan^2\alpha)Tg + (1+\tan^2\alpha) T^2 g] \\
	&= -\frac{1}{\rho^2} \partial_\omega^2 \adj g + \adj [-\tan\alpha(1+\tan^2\alpha)Tg + (1+\tan^2\alpha) T^2 g].
    \end{align*}
    Plugging this relation into the right hand side of \eqref{eq:ttmp}, we obtain
    \begin{align*}
	\left[(1-\rho^2) \partial_\rho^2 + \left(\frac{1}{\rho}-3\rho\right)\partial_\rho\right] \adj g = - \frac{1}{\rho^2}\partial_\omega^2 \adj g + \adj [ (T^2 + 2\tan \alpha T)g], 
    \end{align*}
    hence \eqref{eq:inter1} is proved. Equation \eqref{eq:inter2} follows immediately once noticing that 
    \begin{align*}
	D = \frac{1}{\mu} T^2 \mu + 1,
    \end{align*}
    thus Theorem \ref{thm:intertw} is proved. 
\end{proof}

An integration by parts with zero boundary terms \F{(notice that $\rho$ and $1-\rho^2$ both vanish at the ends of $[0,1]$)} shows that for all \F{$u,v\in C^\infty(\Dm)$}, 
\begin{align}
    (\DD u, v)_{L^2(\Dm)} = \int_\Dm \left( (1-\rho^2) (\partial_\rho u) \overline{(\partial_\rho v)}  + \frac{1}{\rho^2} (\partial_\omega u)\overline{\partial_\omega v} \right)\ \rho\ d\rho\ d\omega + (u, v)_{L^2(\Dm)},
    \label{eq:Luv}    
\end{align}
in particular $\DD$ is formally self-adjoint on $L^2(\Dm)$. In addition, the operator $-T^2$ is formally self-adjoint on $L_+^2(\partial_+ SM)$ or $C^\infty_{\alpha,-,+}(\partial_+ SM)$ defined in \eqref{eq:Calmp}. Indeed, following notation in \cite{Mishra2019}, an orthogonal basis of $L_+^2(\partial_+ S\Dm)$ whose $C^\infty$ span gives $C^\infty_{\alpha,-,+}(\partial_+ SM)$ is given by 
\begin{align}
    \psi_{n,k} := \frac{(-1)^n}{4\pi} e^{i(n-2k)(\beta+\alpha)} (e^{i(n+1)\alpha} + (-1)^n e^{-i(n+1)\alpha}), \qquad n\ge 0, \quad k\in \Zm,
    \label{eq:psink}
\end{align}
and such that $(-T^2) \psi_{n,k} = (n+1)^2 \psi_{n,k}$ for all $n,k$. 

From these observations, passing to the adjoints in \eqref{eq:inter2}, the further intertwining property holds
\begin{align}
    I_0 \circ \DD = (-T^2) \circ I_0.
    \label{eq:inter3}
\end{align}

\subsection{Functional relation for $\adjmu I_0$} \label{sec:funcrel}

In addition to the observations made at the end of last section, we also define
\begin{align}
    Z_{n,k} = \adjmu \psi_{n,k}, \qquad n\ge 0, \qquad 0\le k\le n,
    \label{eq:Znk}
\end{align}
and note that $\adjmu \psi_{n,k} = 0$ for $k\notin [0,n]$. As proved in \cite{Mishra2019}, such functions coincide with the Zernike polynomials in the convention of \cite{Kazantsev2004}, and in light of the results of the previous section, we provide a short proof of the \F{Singular Value Decomposition (SVD)} of $I_0$ \F{(see \cite[Introduction]{Maass1991} for a succinct definition of the SVD of a linear operator, and the benefits of knowing it when approaching a linear inverse problem)}. This SVD has been known for quite some time, see e.g. \cite{Cormack1964,Louis1984}, and the idea to use intertwining differential operators for such derivations can be found e.g. in \cite{Maass1991}, though they are usually written there for each polar harmonic number separately. Equation \eqref{eq:inter2} allows to avoid this separation by harmonics. \F{Below, the ``hat" notation stands for vector normalization in their respective spaces.}

\begin{theorem}\label{thm:SVD}
    The Singular Value Decomposition of $I_0\colon L^2(\Dm) \to L^2 (\partial_+ S\Dm, d\Sigma^2)$ is given by 
    \begin{align}
	(\widehat{Z_{n,k}}, \widehat{\psi_{n,k}}, a_{n,k})_{n\ge 0, 0\le k\le n}, \qquad a_{n,k} := \frac{\sqrt{4\pi}}{\sqrt{n+1}}. 
	\label{eq:ank}
    \end{align}    
\end{theorem}

\begin{proof}
    We obviously have $(-T^2) \psi_{n,k} = (n+1)^2 \psi_{n,k}$ and $-i\partial_\beta \psi_{n,k} = (n-2k) \psi_{n,k}$, which by self-adjointness on $L^2(\partial_+ S\Dm, d\Sigma^2)$ of the two operators applied, makes $\psi_{n,k}$ and orthogonal system. In addition, an immediate computation gives
    \begin{align*}
	\|\psi_{n,k}\|_{L^2(\partial_+ SM)}^2 = \frac{1}{4},\qquad n\ge 0, \qquad k\in \Zm.
    \end{align*}
    In addition we have, as explained in \cite{Mishra2019} $\adjmu \psi_{n,k} = 0$ for $k<0$ or $k>n$, and for $0\le k\le n$, we define $Z_{n,k} := \adjmu \psi_{n,k}$. By Theorem \ref{thm:intertw}, we compute 
    \begin{align*}
	\DD Z_{n,k} &= \DD \adjmu \psi_{n,k} = \adjmu (-T^2) \psi_{n,k} = (n+1)^2 Z_{n,k} \\
	-i\partial_\omega Z_{n,k} &= (n-2k) Z_{n,k},
    \end{align*}
    which immediately makes them an orthogonal system in $L^2(\Dm)$. This gives us orthogonal systems associated with $I_0$ and $\adjmu$ and to compute the singular values, it suffices to normalize all vectors. By definition we have 
    \begin{align*}
	\adjmu\ \widehat{\psi_{n,k}} = a_{n,k}\ \widehat{Z_{n,k}}, \qquad a_{n,k}:= \frac{\|Z_{n,k}\|_{L^2(\Dm)}}{\|\psi_{n,k}\|_{L^2(\partial_+ S\Dm)}},
    \end{align*}
    The SVD for $I_0$ then becomes $(\widehat{Z_{n,k}}, \widehat{\psi_{n,k}}, a_{n,k})$. To compute $a_{n,k}$ \F{it} is given in \cite{Kazantsev2004} that
    \begin{align}
	\|Z_{n,k}\|^2 = \frac{\pi}{n+1}, \qquad n\ge 0, \qquad 0\le k\le n.
	\label{eq:Znknorm}
    \end{align}
    While it is given without proof in \cite{Kazantsev2004} and may rely on properties of orthogonal polynomials or generating functions, we give a functional-analytic proof in the Appendix. The expression of $a_{n,k}$ given in \eqref{eq:ank} then follows. 
\end{proof}

While it is unclear whether the statement that follows is written explicitly in the literature, the ingredients for the proof were known since Zernike's seminal paper \cite{Zernike1934}.

\begin{theorem} \label{thm:relation}
    The following relation holds: 
    \begin{align*}
	\DD (\adjmu I_0)^2 = (4\pi)^2 Id.	
    \end{align*}    
\end{theorem}

\begin{proof} The proof is seen at the level of the spectral decomposition, since we have for every $n\ge 0$ and $0\le k\le n$, 
    \begin{align*}
	\adjmu I_0 Z_{n,k} = \frac{4\pi}{n+1} Z_{n,k}, \qquad \text{ and } \qquad \DD Z_{n,k} = (n+1)^2 Z_{n,k}.
    \end{align*}    
\end{proof}

\subsection{Properties of the operators $\DD$ and $-T^2$} \label{sec:LT}

The operator $\DD$ can be defined on $C^\infty(\Dm)$ (a dense domain in $L^2(\Dm)$), and for any $u\in C^\infty(\Dm)$, we have 
\begin{align}
	(\DD u, u) &= - \int_{\Dm} \left[\frac{1}{\rho} \partial_\rho \left(\rho^2 (1-\rho^2) \frac{1}{\rho} \partial_\rho u \right) + \frac{1}{\rho^2} \partial_\omega^2 u \right]\bar{u}  \rho\ d\rho d\omega + \|u\|_{L^2}^2 \nonumber \\
	&= \int_\Dm \left( (1-\rho^2) |\partial_\rho u|^2 + \frac{1}{\rho^2} |\partial_\omega u|^2 \right)\ \rho\ d\rho\ d\omega + \|u\|_{L^2}^2 \label{eq:Luu}\\
	&\ge \|u\|_{L^2}^2. \nonumber
\end{align}
From the $L^2(\Dm)$-completeness of the Zernike polynomials and the spectral action of $\DD$, we can immediately state the following facts: upon defining the space
\begin{align*}
    D(\DD) &= \left\{ f\in L^2(\Dm),\ \DD f \in L^2(\Dm) \right\} \\
    &= \left\{ f = \sum_{n,k} f_{n,k} \widehat{Z_{n,k}}, \qquad \sum_{n,k} (n+1)^4 |f_{n,k}|^2 <\infty \right\},
\end{align*}
the operator $\DD \colon D (\DD) \to L^2(\Dm)$ is an unbounded self-adjoint operator, with spectrum
\begin{align*}
    \text{sp } \DD = \{ (n+1)^2, \qquad n\in \Nm_0 \},
\end{align*}
with $(n+1)^2$ having multiplicity $n+1$. In particular, we have the property 
\begin{align}
    \| \DD f\|_{L^2(\Dm)} \ge \|f\|_{L^2(\Dm)}, \qquad \forall f\in D(\DD), 
\end{align}
with equality if and only if $f$ is constant. 

\paragraph{A Sobolev-Zernike scale.} For $s\in \Rm$, let us define the scale of Hilbert spaces 
\begin{align}
    \begin{split}
	\wtH^s(\Dm) &= \left\{ f = \sum_{n=0}^\infty\sum_{k=0}^n f_{n,k} \widehat{Z_{n,k}}, \qquad \sum_{n=0}^\infty (n+1)^{2s} \sum_{k=0}^n |f_{n,k}|^2 <\infty \right\} \\
	&= \left\{ f \in L^2(\Dm),\ \DD^{s/2}f \in L^2(\Dm)\right\},
    \end{split}
    \label{eq:Zernikev}
\end{align}
with continuous injections $\wtH^s \subset \wtH^t$ for $s>t$. An important property of the scale $\{\wtH^s(\Dm)\}_s$ is the following: 

\begin{theorem}\label{thm:intersect}
    \begin{align*}
	\bigcap_{s\in \Rm} \wtH^s (\Dm) = C^\infty(\Dm)
    \end{align*}
\end{theorem}

\begin{proof} The inclusion $\supset$ is clear, since a smooth function $f$ is such that for all $n\ge 0$, $\DD^n f\in L^2(\Dm)$. The proof of the inclusion $\subset$ is based on the next two lemmas, proved in the Appendix \ref{sec:missing}.     
    \begin{lemma}\label{lem:continuous} For all $\alpha > 3/2$, we have the continuous injection $\wtH^\alpha (\Dm) \to C(\Dm)$.
    \end{lemma}

    \begin{lemma}\label{lem:tame}
	There exists $\ell>0$ such that for every $\alpha\ge \ell$, the operators 
	\begin{align*}
	    \partial \colon \wtH^\alpha (\Dm)\to \wtH^{\alpha-\ell}(\Dm) \qquad \text{and} \qquad \dbar \colon \wtH^\alpha (\Dm)\to \wtH^{\alpha-\ell}(\Dm)
	\end{align*}
	are bounded. The index $\ell$ can be chosen as $2+\varepsilon$ for every $\varepsilon>0$.
    \end{lemma}
    (Note that the threshold $\ell$ in the previous lemma may not be sharp, though it is enough for the present purposes.)
    
    To prove the inclusion $\subset$, it is enough to show that if $f\in \cap_{s\ge 0} \wtH^s(\Dm)$, then for any $p,q\ge 0$, $\partial^p \dbar^q f\in C(\Dm)$. With $\ell$ a constant as in Lemma \ref{lem:tame}, since $f\in \wtH^{(p+q)\ell + 3}(\Dm)$, repeated use of Lemma \ref{lem:tame} gives that $\partial^p \dbar^q f\in \wtH^3(\Dm)$, and by Lemma \ref{lem:continuous}, this implies that $\partial^p \dbar^q f\in C(\Dm)$. Hence the result.      
\end{proof}

At this point, the conclusions of Theorem \ref{thm:main1}, Lemma \ref{lem:intersect} and Corollary \ref{cor:1} all hold for the Euclidean unit disk. 

\F{

\subsection{Mapping properties involving classical Sobolev scales}

In light of the result above, one might wonder how to tie these mapping estimates with more classical Sobolev scales (``classical'' in the sense that they are modeled over an elliptic $\Psi$DO). We now state two consequences describing what happens when using classical Sobolev spaces on the domain, or on the codomain. Define $H^1(\Dm)$ to be
\begin{align}
    H^1(\Dm):= \left\{u\in L^2(\Dm),\ (u,u)_{L^2(\Dm)} + \int_{\Dm} |\nabla u|^2 <\infty\right\}.
    \label{eq:H1}
\end{align}

\begin{proposition} \label{rmk:SZ}
    The operator $I_0^* I_0$ is not bounded from $L^2(\Dm)$ to $H^1(\Dm)$. 
\end{proposition}

The above means in particular that when incorporating behavior at the boundary, the scales of spaces required need not be the classical ones, in spite of the fact that $I_0^* I_0$ is a classical elliptic $\Psi$DO of order $-1$ in the interior of $\Dm$. This is in stark contrast with the non-compact Euclidean case. 

\begin{proof}[Proof of Proposition \ref{rmk:SZ}]
    From \eqref{eq:Luu}, we immediately see that
    \begin{align*}
	\|u\|_{\wtH^1(\Dm)} = (\DD u, u)_{L^2(\Dm)} \lesssim \|u\|_{H^1(\Dm)}, 
    \end{align*}
    so that $H^1(\Dm) \subset \wtH^1(\Dm)$. This inclusion is strict however, as can be seen from the following calculation. Using \eqref{eq:diffZnk} and orthogonality of the Zernike basis,
    \begin{align*}
	\|\partial_z Z_{n,k}\|^2 = \sum_{p=0}^{P_{n,k}} (n-2p)^2 \|Z_{n-1-2p,k-p}\|^2 &= \pi \sum_{p=0}^{P_{n,k}} (n-2p) \\
	&= \pi (P_{n,k}+1)(n - P_{n,k}),
    \end{align*}
    with $P_{n,k}$ defined in \eqref{eq:Pnk}. In particular, for $n$ even and $k = \frac{n}{2}$, $P_{n,\frac{n}{2}} = \frac{n}{2}-1$ and thus 
    \begin{align*}
	\|Z_{n,k}\|^2_{H^1} \ge \|\partial_z Z_{n,k}\|_{L^2}^2 = \pi \frac{n(n+2)}{4}
    \end{align*}
    On the other hand, we have 
    \begin{align*}
	\|Z_{n,k}\|_{\wtH^1} = \|\DD^{1/2} Z_{n,k}\|_{L^2} = (n+1) \|Z_{n,k}\|_{L^2} \stackrel{\eqref{eq:Znknorm}}{=} \pi \sqrt{n+1},
    \end{align*}
    so $\sup_{n,k} \|Z_{n,k}\|_{H^1}/\|Z_{n,k}\|_{\wtH^1} = \infty$.

    By Corollary \ref{cor:1}, since $\adjmu I_0 (L^2(\Dm)) = \wtH^1(\Dm) \supsetneq H^1(\Dm)$, this implies the unboundedness of the operator $\adjmu I_0 \colon L^2(M) \to H^1(\Dm)$.    
\end{proof}

Finally, we discuss the possibility of writing continuity statements for $I_0$ on scales of spaces on $\partial_+ SM$ which control regularity along all directions. To do so, first note that every vector field on $\partial_+ SM$ is a linear combination of $T = \partial_\beta-\partial_\alpha$ and $\partial_\beta$, so the classical Sobolev scale on $\partial_+ SM$ can be defined using the elliptic operator $-(T^2 + \partial_\beta^2)$, which also acts diagonally on the $\psi_{n,k}$ basis of $L^2_+(\partial_+ SM)$ as 
\begin{align*}
    -(T^2 + \partial_\beta^2) \psi_{n,k} = ( (n+1)^2 + (n-2k)^2) \psi_{n,k}, \qquad n\ge 0, \qquad k\in \Zm.  
\end{align*}
One may then define 
\begin{align}
    H_+^{s} (\partial_+ SM) := \{u\in L^2_+ (\partial_+ SM), \quad (-(T^2 + \partial_\beta^2))^{s/2} u \in L^2_+ (\partial_+ SM)\}, \qquad s\ge 0.
    \label{eq:classic}
\end{align}

\begin{proposition}\label{rmk:otherscales} Fix any $s\ge 0$. Then the following hold: 

    (i) $H^s_{+} (\partial_+ SM) \subsetneq H^s_{T,+}(\partial_+ SM)$.  

    (ii) The topologies $H_+^s(\partial_+ SM)$ and $H_{T,+}^s (\partial_+ SM)$ are equivalent on the range of $I_0$. In particular, the operator $I_0\colon \wtH^s(\Dm)\to H_+^{s+1/2} (\partial_+ SM)$ is bounded.
    
\end{proposition}

\begin{proof}
    To prove (i), it suffices to notice the equality
    \begin{align}
	\frac{\|\psi_{n,k}\|_{H^s_{T,+}}}{\|\psi_{n,k}\|_{H^s_+}} = \left( 1 + \left(\frac{n-2k}{n+1}\right)^2 \right)^{-s/2}, \qquad n\ge 0, \quad k\in\Zm, \quad s\ge 0. 
	\label{eq:ratio}
    \end{align}
    Since the right hand side is bounded above by $1$, this implies $H^s_{+} (\partial_+ SM) \subset H^s_{T,+}(\partial_+ SM)$. However the converse inclusion is not true in general since the above right hand side can become arbitrarily close to zero as $k\to \infty$ while keeping $n$ fixed.

    To prove (ii), notice that on the range of $I_0$, spanned by those $\psi_{n,k}$ for which we have $|n-2k|\le n$, the right hand side in \eqref{eq:ratio} becomes bounded below by $2^{-s/2}$. The continuity statement follows by combining the estimate \eqref{eq:sharp} with the fact that the inclusion $H^{s+1/2}_{T,+}(\partial_+ SM) \to H^{s+1/2}_{+} (\partial_+ SM)$ is bounded when restricted to the range of $I_0$.
\end{proof}








}

\section{Simple geodesic disks of constant curvature} \label{sec:CCD}

Given $\kappa\in \Rm$ and $R>0$ such that $R^2 |\kappa|<1$, let us now equip $\Dm_R:=\{(x,y)\in \Rm^2: x^2+y^2\le R^2\}$ with the metric $g_\kappa(z) := \left( 1 +\kappa|z|^2 \right)^{-2} |dz|^2$, of constant curvature $4\kappa$. We denote $S_{(\kappa)}\Dm_R$ the unit tangent bundle
\begin{align*}
    S_{(\kappa)}\Dm_R = \{(x,v) \in T\Dm_R,\quad (g_\kappa)_x(v,v) = 1\},
\end{align*}
with inward boundary $\partial_+ S_{(\kappa)}\Dm_R$ defined as usual. The latter is parameterized in fan-beam coordinates $(\beta,\alpha)\in \Sm^1\times (-\pi/2,\pi/2)$, where $\beta$ describes the boundary point $x = R e^{i\beta}$, and $\alpha$ describes the angle of the tangent vector with respect to the unit inner normal $\nu_x$, i.e. $v = (1+R^2\kappa) e^{i(\beta+\pi+\alpha)}$. The manifold $\partial_+ S_{(\kappa)}\Dm_R$ is a model for all geodesics on $\Dm_R$ intersecting $\partial \Dm_R$ transversally, equipped with the measure $d\Sigma^2 = R (1+R^2\kappa)^{-1} d\beta\ d\alpha$.

\subsection{Intertwining diffeomorphisms}

Consider the X-Ray transform $I_0 \colon L^2(\Dm_R, \dvolk) \to L^2 (\partial_+ S_{(\kappa)}\Dm_R, d\Sigma^2)$
\begin{align*}
    I_0 f (\beta,\alpha) := \int_0^{\tau(\alpha)} f(\gamma_{\beta,\alpha}(t))\ dt,
\end{align*}
where $\gamma_{\beta,\alpha}$ is the unit-speed $g_\kappa$-geodesic passing through $(R e^{i\beta}, c_\kappa(R) e^{i(\beta+\alpha+\pi)})$, and where $\tau(\alpha)$ is its exit time out of $M$.

Given a point $(\rho e^{i\omega}, \theta)\in S\Dm$, denote $(\beta_-(\rho e^{i\omega}, \theta), \alpha_- (\rho e^{i\omega}, \theta))$ the fan-beam coordinates of the unique unit-speed geodesic passing through $(\rho e^{i\omega}, \theta)$, or 'footpoint map'. Our reference case is for $\kappa = 0$ and $R=1$, for which we will denote $I_0^e$ and $(\beta_-^e, \alpha_-^e)$. In particular, $(\beta_-^e,\alpha_-^e)$ are uniquely defined by the relations 
\begin{align}
    \beta^e_-(\rho,\theta) + \alpha^e_- (\rho,\theta) + \pi = \theta, \qquad \sin \alpha^e_- (\rho,\theta) = -\rho\sin\theta, \qquad \rho\in [0,1],\quad \theta\in \Sm^1.
    \label{eq:btale}
\end{align}
The adjoint of $I_0\colon L^2(\Dm, \dvolk)\to L^2_\mu (\partial_+S\Dm)$ is now defined as
\begin{align*}
    I_0^{\sharp} h (\rho e^{i\omega}) &:= \int_{\Sm^1} h(\beta_-(\rho e^{i\omega}, \theta), \alpha_-(\rho e^{i\omega}, \theta))\ d\theta \\
    &= \int_{\Sm^1} h(\omega + \beta_-(\rho, \theta), \alpha_-(\rho, \theta))\ d\theta, 
\end{align*}
where the last equality follows from the symmetry property
\begin{align*}
    \beta_-(\rho e^{i\omega}, \theta) = \beta_-(\rho,\theta-\omega) + \omega,\qquad \alpha_- (\rho e^{i\omega}, \theta) = \alpha_-(\rho, \theta-\omega).
\end{align*}
And we have the relation $\adjmu = \adj \frac{1}{\mu}$ between the adjoints for the different codomain topologies.

In the recent article \cite{Mishra2019}, it was proved that the SVD of $\adjmu$ for the case $R=1$ could be obtained from the SVD of $(\adjmu)^e$ via specific changes of variables. We now make this relation hold directly at the level of the operators and show that this actually holds for any $\kappa$ and $R$ such that $R^2 |\kappa|<1$. Define the map $\ss\colon \Sm^1 \to \Sm^1$
\begin{align}
    \ss (\alpha) := \tan^{-1} \left( \frac{1-R^2\kappa}{1+R^2\kappa}\tan \alpha \right),
    \label{eq:skappa}
\end{align}
first defined for $|\alpha|\le \frac{\pi}{2}$ and extended as a $\pi$-periodic function. We can regard $\ss$ as a map $\ss\colon \partial S\Dm_R \to \partial S\Dm_R$ where, abusing notation $\ss(\beta,\alpha) := (\beta,\ss(\alpha))$. This map is such that the scattering relation $\SS$ and antipodal scattering relation $\SS_A$ for $(\Dm_R, g_\kappa)$ are given by 
\begin{align*}
    \SS(\beta,\alpha) = (\beta+\pi + 2\ss(\alpha), \pi-\alpha), \qquad \SS_A(\beta,\alpha) = (\beta+\pi + 2\ss(\alpha), -\alpha). 
\end{align*}
The proof of this is a similar calculation to \cite[Section 2.2]{Mishra2019}: the $g_\kappa$-geodesic passing through the point $(R, (1+\kappa R^2) e^{i(\pi+\alpha)})\in \partial_+ S\Dm_R$ takes the (non-unit speed) form $T(x) = \frac{R-x e^{i\alpha}}{1+R\kappa e^{i\alpha} x}$ for $x$ in some real open interval. Solving $|T(x)| = R$ gives two roots $x=0$ and $x^* >0$, and one finds that $T(x^*) = R e^{i (\pi + 2\ss(\alpha))}$ with $\ss(\alpha)$ defined in \eqref{eq:skappa}. 

The following result contains many of the tedious calculations required. Define the map $\Psi \colon [0,R]\times \Sm^1 \to [0,1]\times \Sm^1$ by 
\begin{align}
    (\rho,\theta) \mapsto \left( \rho' = \Phi(\rho) := \frac{1-\kappa R^2}{1-\kappa\rho^2} \frac{\rho}{R}, \quad \theta' = \theta - \tan^{-1} \left(\frac{\kappa\rho^2 \sin(2\theta)}{1+\kappa\rho^2 \cos(2\theta)}\right) \right)
    \label{eq:Psi}
\end{align}
$\Phi$ can also be thought of as the diffeomorphism $\Phi \colon \Dm_R\to \Dm_1$ upon defining $\Phi (\rho e^{i\omega}) := \rho' e^{i\omega}$. Similarly, one should think of $\Psi$, augmented accordingly, as a global diffeomorphism from $S_{(\kappa)} \Dm_R$ onto $S_{(0)}\Dm_1$ given by $\Psi (\rho e^{i\omega}, \theta) = (\rho' e^{i\omega}, \omega + \theta')$.

\begin{lemma}\label{lem:crux}
    With $\Psi\colon (\rho,\theta)\to (\rho',\theta')$, $\Phi$, $\ss$ defined as above, we have 
    \begin{align}
	(\beta_- (\rho,\theta), \ss(\alpha_-(\rho,\theta))) = (\beta_-^e (\rho',\theta'), \alpha_-^e (\rho',\theta')), \qquad \forall (\rho,\theta)\in [0,R]\times \Sm^1.
	\label{eq:geometries}
    \end{align}
    In short, we have $\ss \circ F = F_e \circ \Psi$, where $F\colon S_{(\kappa)}\Dm_R \to \partial_+ S_{(\kappa)}\Dm_R$ and $F_e \colon S_{(0)} \Dm_1\to \partial_+ S_{(0)} \Dm_1$ denote the footpoint maps. In addition, the following relation holds: 
    \begin{align}
	\frac{\partial \theta'}{\partial\theta} = \frac{1-\kappa\rho^2}{1+\kappa\rho^2} \frac{1+\kappa R^2}{1-\kappa R^2}\ \ss'(\alpha_-(\rho,\theta)), \label{eq:jactheta}
    \end{align}    
\end{lemma}

\begin{proof} Given that the Euclidean footpoint map is uniquely determined by the relations \eqref{eq:btale}, equation \eqref{eq:geometries} will be established once we can show that 
    \begin{align}
	\beta_-(\rho,\theta) + \ss (\alpha_-(\rho,\theta)) + \pi = \theta', \qquad \sin (\ss(\alpha_-(\rho,\theta))) = -\rho' \sin\theta'.
	\label{eq:wts}
    \end{align}
    To this end, we first prove that 
    \begin{align}
	\beta_-(\rho,\theta) + \ss(\alpha_-(\rho,\theta)) + \pi = \theta - \tan^{-1} \left(\frac{\kappa\rho^2 \sin(2\theta)}{1+\kappa\rho^2 \cos(2\theta)}\right).
	\label{eq:bmam}
    \end{align}

    \begin{proof}[Proof of \eqref{eq:bmam}]
	The proof is similar to \cite[Lemma 13]{Mishra2019}, done here for general $R$. Given $(\rho,\theta)$, the unique $g_\kappa$-geodesic passing through $(\rho, c_\kappa(\rho)e^{i\theta})$ takes the form $T(x) = \frac{e^{i\theta} x + \rho}{1-\kappa e^{i\theta} \rho x}$ for $x$ in an open real interval. The endpoints $x_{\pm}$ are solved for \F{by} writing $|T(x)|^2 = R^2$. Note also that $T(x_-) = R e^{i\beta_-}$ and $T(x_+) = R e^{i(\beta_- + 2\ss(\alpha_-) + \pi)}$, so that, computing $T(x_-) T(x_+)$ in two ways, one obtains the relation
	\begin{align*}
	    -R^2 e^{2i(\beta_- + \ss(\alpha_-))} = - R^2 e^{2i\theta} \frac{1+ \kappa\rho^2 e^{-2i\theta}}{1+\kappa\rho^2 e^{2i\theta}}.
	\end{align*}
	One then deduces \eqref{eq:bmam} by comparing arguments, and using that in the test case, the offset of $\pi$ in the left hand side of \eqref{eq:bmam} is determined from looking at the reference case. 
    \end{proof}

    We next prove that 
    \begin{align}
	\sin(\alpha_-) = - \frac{1+\kappa R^2}{1+\kappa \rho^2} \frac{\rho}{R} \sin\theta
	\label{eq:sinam1}
    \end{align}
    
    \begin{proof}[Proof of \eqref{eq:sinam1}] The proof is similar to \cite[Lemma 13]{Mishra2019}, done here for general $R$. Upon defining, for $\kappa\in \Rm$, 
	\begin{align*}
	    \sin_{4\kappa} (x) := x - \frac{(4\kappa)x^3}{3!} + \frac{(4\kappa)^2 x^4}{5!} - \frac{(4\kappa)^3 x^7}{7!} + \dots,
	\end{align*}
	one may deduce using trigonometric identities that $\sin_{4\kappa} (d_\kappa(\rho,o)) = \frac{\rho}{1+\kappa\rho^2}$. Next apply the generalized law of sines to the geodesic triangle with vertices $o$, $\rho$ and $R e^{i\beta_- (\rho,\theta)}$, to make appear
	\begin{align*}
	    \frac{\sin(-\alpha_- (\rho,\theta))}{\sin_{4\kappa} (d_\kappa(\rho,\theta))} = \frac{\sin\theta}{ \sin_{4\kappa} (d_\kappa (R,0))}. 
	\end{align*}
	Equality \eqref{eq:sinam1} follows.
    \end{proof}

    \F{To simplify the equations below, we substitute} $\lambda := R^2 \kappa$ and $\rho_R := \frac{\rho}{R}$. In particular, \eqref{eq:sinam1} reads
    \begin{align}
	\sin(\alpha_-) = - \frac{1+\lambda}{1+\lambda \rho_R^2} \rho_R \sin\theta. 
	\label{eq:sinam2}
    \end{align}

    To obtain an equation for $\sin(\ss(\alpha_-))$ instead of for $\sin \alpha_-$ in \eqref{eq:sinam2}, notice that the relation between $\sin(\alpha)$ and $\sin (\ss(\alpha))$ is the same as \cite{Mishra2019} upon substituting $\kappa$ into $\lambda$. The following identities are thus the same calculation as to obtain \cite[Eq. (15)]{Mishra2019}, which now reads: 
    \begin{align}
	\sin(\ss(\alpha)) = \sqrt{\frac{1-\lambda}{1+\lambda}} \sqrt{\ss'(\alpha)} \sin \alpha, \qquad \cos(\ss(\alpha)) = \sqrt{\frac{1-\lambda}{1+\lambda}} \sqrt{\ss'(\alpha)} \cos \alpha, \qquad \lambda = \kappa R^2.
	\label{eq:trig}
    \end{align}
    The identity for the sines combined with \eqref{eq:sinam2} gives
    \begin{align}
	\frac{\sin (\ss(\alpha_-))}{\sqrt{\ss'(\alpha_-)}} = -\frac{\sqrt{1-\lambda^2}}{1+\lambda\rho_R^2} \rho_R \sin \theta. 	
	\label{eq:sinam}
    \end{align}
    We move to the properties of the fiber variable $\theta'$ given by
    \begin{align}
	\theta'(\rho,\theta) := \theta - \tan^{-1} \left(\frac{\kappa\rho^2 \sin(2\theta)}{1+\kappa\rho^2 \cos(2\theta)}\right) = \theta - \tan^{-1} \left(\frac{\lambda\rho_R^2 \sin(2\theta)}{1+\lambda\rho_R^2 \cos(2\theta)}\right),
	\label{eq:thetap}
    \end{align}
    In the case $R = 1$, the Jacobian $\frac{\partial\theta'}{\partial \theta}$ is computed in \cite[Lemma 14]{Mishra2019} as well as an identity relating $\sin\theta$ and $\sin\theta'$. To obtain the present case of general $R$, it suffices to notice that all derivations are formally identical upon changing $(\kappa,\rho)$ into $(\lambda,\rho_R)$. One thus finds that
    \begin{align}
	\frac{\partial \theta'}{\partial\theta} &= \frac{1-\lambda\rho_R^2}{1+\lambda\rho_R^2} \frac{1+\lambda}{1-\lambda} \ss'(\alpha_-), \label{eq:jacthetat} \\
	\sin\theta' &= \frac{1-\lambda\rho_R^2}{1+\lambda\rho_R^2} \sqrt{\frac{1+\lambda}{1-\lambda}} \sqrt{ \ss'(\alpha_-)} \sin\theta. \label{eq:sintheta}
    \end{align}
    From \eqref{eq:jacthetat}, Equation \eqref{eq:jactheta} follows. Combining \eqref{eq:sintheta} with \eqref{eq:sinam}, we also have the relation
    \begin{align*}
	\sin(\ss(\alpha_-)) = - \frac{1-\lambda}{1-\lambda\rho_R^2} \rho_R\sin \theta' = -\rho' \sin\theta'.
    \end{align*}
    Together with \eqref{eq:bmam} and the definition of $\theta'$, we see that \eqref{eq:wts} is fulfilled and thus Lemma \ref{lem:crux} is proved.     
\end{proof}

The diffeomorphisms above are essentially all we need to show that $\adjmu$ and $(\adjmu)^e$ are in fact intertwined via changes of variables and their jacobians. Below, given a smooth diffeomorphism $h:X\to Y$, we denote $h^*$ the induced {\em pull-back operator} $C^\infty(Y)\ni f \mapsto h^*f = f\circ h\in C^\infty(X)$.

\begin{theorem}\label{thm:intertw2}
    Fix $\kappa\in\Rm$ and $R>0$ such that $R^2|\kappa|<1$, and define $\Phi:\Dm_R \to \Dm_1$ as $\Phi(z) := \frac{1-\kappa R^2}{1-\kappa|z|^2}\frac{z}{R}$ as well as $w(z) := \frac{1+\kappa|z|^2}{1-\kappa|z|^2}$. Then we have the following intertwining relation between $\adjmu$ and $(\adjmu)^e$:
    \begin{align}
	\sqrt{\frac{1+\kappa R^2}{1-\kappa R^2}} (\Phi^{-1})^* \frac{1}{w} \adjmu \sqrt{\ss'} \ss^* = (\adjmu)^e.
	\label{eq:intertw2}
    \end{align}
    Passing to the adjoints, 
    \begin{align}
	I_0^e = \frac{(1-\kappa R^2)^{3/2}}{R (1+\kappa R^2)^{1/2}} (\ss^{-1})^* \frac{1}{\sqrt{\ss'}} I_0 w^2 \Phi^*.
	\label{eq:intertw_adj}
    \end{align}
\end{theorem}

\begin{proof} {\bf Proof of \eqref{eq:intertw2}.} First note that the second identity in \eqref{eq:trig} can be written, in terms of $\mu = \cos\alpha$, as
    \begin{align*}
	\mu\circ \ss = \sqrt{\frac{1-\kappa R^2}{1+\kappa R^2}} \sqrt{\ss'}\ \mu.
    \end{align*}
    With this in mind, we compute that
    \begin{align*}
	\adj \left[\frac{1}{\mu} \sqrt{\ss'} g \right](\rho e^{i\omega}) &= \sqrt{\frac{1+\kappa R^2}{1-\kappa R^2}} \adj \left[ \frac{\ss'}{\mu\circ \ss} g\right] (\rho e^{i\omega})  \\
	&= \sqrt{\frac{1+\kappa R^2}{1-\kappa R^2}} \int_{\Sm^1} \frac{g(\omega+\beta_-(\rho,\theta), \alpha_-(\rho,\theta))}{\cos (\ss(\alpha_-(\rho,\theta)))} \ss'(\alpha_-(\rho,\theta))\ d\theta \\
	&= \sqrt{\frac{1+\kappa R^2}{1-\kappa R^2}} \int_{\Sm^1} \frac{(\ss^{-1})^* g(\omega+\beta_-(\rho,\theta), \ss(\alpha_-(\rho,\theta)))}{\cos (\ss(\alpha_-(\rho,\theta)))} \ss'(\alpha_-(\rho,\theta))\ d\theta
    \end{align*}
    then using \eqref{eq:jactheta}, 
    \begin{align*}
	\frac{1-\kappa\rho^2}{1+\kappa\rho^2} \sqrt{\frac{1+\kappa R^2}{1-\kappa R^2}}\ \adj \left[\frac{1}{\mu} \sqrt{\ss'} g \right](\rho e^{i\omega}) &= \int_{\Sm^1} \frac{(\ss^{-1})^* g(\omega+\beta_-(\rho,\theta), \ss(\alpha_-(\rho,\theta)))}{\cos (\ss(\alpha_-(\rho,\theta)))} \frac{\partial \theta'}{\partial \theta}\ d\theta \\
	&= \int_{\Sm^1} \frac{(\ss^{-1})^* g(\omega+\beta_-(\rho,\theta), \ss(\alpha_-(\rho,\theta)))}{\cos (\ss(\alpha_-(\rho,\theta)))} \ d\theta',
    \end{align*}
    where $\theta$ is implicitly thought of a function of $\theta'$. Now use \eqref{eq:geometries} to obtain 
    \begin{align*}
    \frac{1-\kappa\rho^2}{1+\kappa\rho^2} \sqrt{\frac{1+\kappa R^2}{1-\kappa R^2}}\ \adj \left[\frac{1}{\mu} \sqrt{\ss'} g \right](\rho e^{i\omega}) &= \int_{\Sm^1} \frac{ (\ss^{-1})^* g(\omega+\beta^e_-(\rho',\theta'), \alpha_-^e(\rho',\theta'))}{\cos (\alpha_-^e(\rho',\theta')))} d\theta' \\
    &= (\adj)^e \left[ \frac{1}{\mu} (\ss^{-1})^* g \right] (\rho' e^{i\omega}).
    \end{align*}    
    In conclusion, we obtain at the level of the operators: 
    \begin{align*}
	\sqrt{\frac{1+\kappa R^2}{1-\kappa R^2}} \frac{1}{w} \adjmu \sqrt{\ss'} = \Phi_\kappa^* (\adjmu)^e (\ss^{-1})^*,
    \end{align*}
    which is equivalent to \eqref{eq:intertw2}.

    {\bf Proof of \eqref{eq:intertw_adj}.} To get back to $I_0$, we compute formally
    \begin{align*}
	(g, I_0^e f)_{d\alpha d\beta} &= ((I_0^e)^* g, f)_{\rho' d\rho' d\omega} \\
	&= \sqrt{\frac{1+\kappa R^2}{1-\kappa R^2}} \left( (\Phi^{-1})^* \frac{1}{w} I_0^* \sqrt{\ss'} \ss^* g, f\right)_{\rho' d\rho' d\omega}.
    \end{align*}
    Now writing $\rho' e^{i\omega} = \Phi (\rho e^{i\omega})$, with change of volume 
    \begin{align}
	\rho'\ d\rho'\ d\omega = \frac{(1-\kappa R^2)^2}{R^2} w^{3}(\rho)\ dVol_\kappa (\rho e^{i\omega}),
	\label{eq:changevolume}
    \end{align}
    we obtain
    \begin{align*}
	\left( (\Phi^{-1})^* \frac{1}{w} h, f \right)_{\rho' d\rho' d\omega} = \frac{(1-\kappa R^2)^2}{R^2} \left( h, w^2 \Phi^{*}f \right)_{\dvolk},
    \end{align*}
    and thus, with $h = I_0^* \sqrt{\ss'} \ss^* g$, 
    \begin{align*}
	(g, I_0^e f)_{d\alpha d\beta} &= \sqrt{\frac{1+\kappa R^2}{1-\kappa R^2}} \frac{(1-\kappa R^2)^2}{R^2} \left( I_0^* \sqrt{\ss'} \ss^* g, w^2 \Phi^{*}f \right)_{\dvolk} \\
	&= \sqrt{\frac{1+\kappa R^2}{1-\kappa R^2}} \frac{(1-\kappa R^2)^2}{R^2} \left( \sqrt{\ss'} \ss^* g, I_0 w^2 \Phi^* f  \right)_{\frac{R}{1+\kappa R^2}\ d\beta\ d\alpha} \\
	&= \frac{(1-\kappa R^2)^{3/2}}{R (1+\kappa R^2)^{1/2}} \left( \sqrt{\ss'} \ss^* g, I_0 w^2 \Phi^* f  \right)_{d\beta\ d\alpha} \\
	&= \frac{(1-\kappa R^2)^{3/2}}{R (1+\kappa R^2)^{1/2}} \left( g, (\ss^{-1})^* \frac{1}{\sqrt{\ss'}} I_0 w^2 \Phi^* f  \right)_{d\beta\ d\alpha},
    \end{align*}
    hence the result. 
\end{proof}
 
\subsection{Intertwining operators - proof of Theorem \ref{thm:main1}}

Fix $\kappa\in \Rm$ and $R>0$ such that $R^2 |\kappa|<1$. Define
\begin{align}
    T := \sqrt{\ss'} \ss^* T_e (\ss^{-1})^* \frac{1}{\sqrt{\ss'}}, \qquad \DD := w \Phi^* \DD_e (\Phi^{-1})^* \frac{1}{w}.
    \label{eq:LTkappa}
\end{align}

\begin{proof}[Proof of Theorem \ref{thm:main1} for constant curvature disks]
    The relations appearing in \eqref{eq:rel1} are an immediate consequence of the intertwining relations \eqref{eq:inter2} and \eqref{eq:inter3}, rewritten here as
    \begin{align*}
	\DD_e \circ (\adjmu)^e = (\adjmu)^e \circ (-T_e^2), \qquad I_0^e \circ \DD_e = (-T_e^2) \circ I_0^e,
    \end{align*}
    combined with relations \eqref{eq:intertw2} and \eqref{eq:intertw_adj} and the definition \eqref{eq:LTkappa} of $\DD$ and $T$. 
    
    To prove \eqref{eq:rel2}, insert \eqref{eq:intertw2}, \eqref{eq:intertw_adj} and \eqref{eq:LTkappa} into the relation $\DD_e \left( (I_0^e)^* I_0^e \right)^2 = (4\pi)^2 Id$ to make appear \eqref{eq:rel2}. The proof of Theorem \ref{thm:main1} is complete.      
\end{proof}

We now make the operators $T$ and $\DD$ a bit more explicit, in particular we show that $-T^2$ and $\DD$ are self-adjoint in appropriate spaces. We first compute 
\begin{align*}
    T \circ (\ss^{-1})^* u (\beta,\alpha) &= (\partial_\beta-\partial_\alpha) \left(u (\beta,\ss^{-1}(\alpha))\right) \\
    &= \partial_\beta u(\beta,\ss^{-1}(\alpha)))  - (\ss^{-1})'(\alpha)\partial_\alpha u (\beta,\ss^{-1}(\alpha)) \\
    &= \left[ \left(\partial_\beta - \frac{1}{\ss'}\partial_\alpha\right) u \right] (\beta,\ss^{-1}(\alpha)),
\end{align*}
and thus $\ss^* \circ T\circ (\ss^{-1})^* = \partial_\beta - \frac{1}{\ss'(\alpha)} \partial_\alpha$, which is easily seen to be formally skew-adjoint on $L^2(\partial_+ SM, \ss'(\alpha)\ d\Sigma^2)$. As a result, 
\begin{align}
    T = \sqrt{\ss'} \left(\partial_\beta - \frac{1}{\ss'(\alpha)} \partial_\alpha\right) \frac{1}{\sqrt{\ss'}},
    \label{eq:Tk2}
\end{align}
which is a formally skew-adjoint operator on $L^2(\partial_+ S_{(\kappa)} \Dm_R, d\Sigma^2)$.

On to $\DD$, similar observations show that, since $\DD_e$ is self-adjoint on $L^2(\Dm, \rho\ d\rho\ d\omega)$, we obtain

\begin{lemma}\label{lem:LR_sa}
    The operator $\DD$ defined in \eqref{eq:LTkappa} is formally self-adjoint on $L^2(\Dm_R, w\ \dvolk)$ with $\dvolk = \frac{\rho\ d\rho\ d\omega}{(1+\kappa\rho^2)^2}$.    
\end{lemma}
\begin{proof}
    We compute, using notation $\rho' e^{i\omega} = \Phi (\rho e^{i\omega})$,  
    \begin{align*}
	\int_{\Dm} \DD u (\rho e^{i\omega}) \bar{v} (\rho e^{i\omega}) w \ dVol_\kappa &= \int_\Dm \DD_e \left[(\Phi^{-1})^* \frac{1}{w}u\right] (\rho' e^{i\omega}) \bar{v} (\rho e^{i\omega}) w^2\ dVol_\kappa \\
	&= \int_\Dm \DD_e \left[(\Phi^{-1})^* \frac{1}{w} u\right] (\rho' e^{i\omega}) \left[(\Phi^{-1})^*\frac{1}{w}\bar{v}\right] (\rho' e^{i\omega}) w^3\ dVol_\kappa (\rho e^{i\omega}),
    \end{align*}
    where $\rho' = \frac{1-\kappa R^2}{1-\kappa\rho^2} \frac{\rho}{R}$. Using the change of volume \eqref{eq:changevolume}, we change variable $\rho\to \rho'$ in the last integral and obtain
    \begin{align*}
	\int_{\Dm} \DD u (\rho e^{i\omega}) \bar{v} (\rho e^{i\omega}) w \ dVol_\kappa &= \frac{R^2}{(1-\kappa R^2)^2} \int_\Dm \DD_e \left[(\Phi^{-1})^* \frac{u}{w}\right] (\rho' e^{i\omega}) \left[(\Phi^{-1})^*\frac{\bar{v}}{w}\right] (\rho' e^{i\omega}) \rho'\ d\rho'\ d\omega, 
    \end{align*}
    which is now a symmetric expression of $u$ and $v$ since $\DD_e$ is formally self-adjoint on $L^2(\Dm, |dz|^2)$. The proof is complete.  
\end{proof}

Upon denoting $\psi_{n,k}^e$ and $Z_{n,k}^e$ the functions defined in \eqref{eq:psink} and \eqref{eq:Znk}, from the relations
\begin{align*}
    \DD_e Z_{n,k}^e &= (n+1)^2 Z_{n,k}^e, \qquad n\ge 0, \qquad 0\le k\le n, \\
    (-T_e^2) \psi_{n,k}^e &= (n+1)^2 \psi_{n,k}^e, \qquad n\ge 0, \qquad k\in \Zm,
\end{align*}
we can combine these relations with the definitions \eqref{eq:LTkappa} to deduce the relations
\begin{align}
    \DD Z_{n,k} &= (n+1)^2 Z_{n,k}, \qquad Z_{n,k} := w \Phi^* Z_{n,k}^e, \qquad n\ge 0, \qquad 0\le k\le n, \label{eq:Znk2}\\
    (-T^2) \psi_{n,k} &= (n+1)^2 \psi_{n,k}, \qquad \psi_{n,k} := \sqrt{\ss'} \ss^* \psi^{\F{e}}_{n,k}, \qquad n\ge 0, \qquad k\in \Zm. \label{eq:psink2}
\end{align}
Note also that 
\begin{align}
    \| Z_{n,k}\|_{L^2(\Dm_R, w\ \dvolk)}^2 \stackrel{\eqref{eq:changevolume}}{=} \frac{R^2}{(1-\kappa R^2)^2} \|Z_{n,k}^e\|_{L^2(\Dm)}^2 \stackrel{\eqref{eq:Znknorm}}{=} \frac{R^2}{(1-\kappa R^2)^2} \frac{\pi}{n+1}
    \label{eq:Znknorm2}
\end{align}
and 
\begin{align}
    \|\psi_{n,k}\|_{L^2(\partial_+ S_{(\kappa)}\Dm_R, R(1+\kappa R^2)^{-1} d\beta d\alpha)}^2 = \frac{R}{1+\kappa R^2} \|\psi_{n,k}^e\|^2_{L^2(\partial_+ S\Dm, d\beta d\alpha)} = \frac{R}{1+\kappa R^2} \frac{1}{4}. 
    \label{eq:psinknorm2}
\end{align}

\subsection{Mapping properties of the normal operator $\adjmu I_0$ - Proof of Lemma \ref{lem:intersect}}

With $Z_{n,k}$ defined in \eqref{eq:Znk2}, a function $f\in L^2(\Dm_R, w\ \dvolk)$ decomposes as 
\begin{align*}
    f = \sum_{n\ge 0} \sum_{k=0}^n f_{n,k} \widehat{Z_{n,k}}, \qquad f_{n,k} = (f, \widehat{Z_{n,k}})_{L^2(\Dm_R, w\ \dvolk)} \qquad \|f\|_{L^2(\Dm_R, w\ \dvolk)}^2 = \sum_{n,k} |f_{n,k}|^2.
\end{align*}
In view of the eigenequation \eqref{eq:Znk2}, it may be natural to define the following Sobolev-Zernike scale
\begin{align*}
    \wtH^s(\Dm_R) &= \left\{ f = \sum_{n=0}^\infty\sum_{k=0}^n f_{n,k} \widehat{Z_{n,k}}, \qquad \sum_{n=0}^\infty (n+1)^{2s} \sum_{k=0}^n |f_{n,k}|^2 <\infty \right\} \\
    &= \left\{ f \in L^2(\Dm, w\ \dvolk),\qquad \DD^{s/2}f \in L^2(\Dm_R, w\ \dvolk)\right\}.
\end{align*}
Since $Z_{n,k} = w \Phi^* Z_{n,k}^e$, the following claim is immediate for any $s\in \Rm$
\begin{align}
    f\in \wtH_e^s(\Dm) \qquad \text{if and only if} \qquad  w\ \Phi^* f \in \wtH^s(\Dm_R),
    \label{eq:claim}
\end{align}
where $\wtH_e^s$ denotes the space of the reference case. This allows to easily prove Lemma \ref{lem:intersect} for the case of simple geodesic disks of constant curvature, indeed 
\begin{align*}
    \bigcap_{s\in \Rm} \wtH^s (\Dm_R) = \bigcap_{s\in \Rm} (\Phi^{-1})^* \frac{1}{w} \wtH_e^s (\Dm) = (\Phi^{-1})^* \frac{1}{w} \bigcap_{s\in \Rm} \wtH_e^s (\Dm) \stackrel{(\star)}{=} (\Phi^{-1})^* \frac{1}{w} C^\infty(\Dm) = C^\infty(\Dm_R),
\end{align*}
where the equality $(\star)$ uses that Lemma \ref{lem:intersect} is true for the reference case.
As a result, the conclusion of Lemma \ref{lem:intersect}, and {\it a fortiori} of Corollary \ref{cor:1} holds for simple geodesic disks of constant curvature. 

\section{Mapping properties of $I_0$ - Proof of Theorem \ref{thm:main2}} \label{sec:mappingI0}

We now discuss a natural Sobolev scale on $\partial_+ SM$, where $(M,g)$ is modeled on $(\Dm_R, g_\kappa)$ with $R^2 |\kappa|<1$. Recall that the space $L^2(\partial_+ SM)$ splits into a direct orthogonal sum 
\begin{align*}
    L^2(\partial_+ SM) = L^2_+ \stackrel{\perp}{\oplus} L^2_-, \qquad L^2_{\pm} (\partial_+ SM) = \{u\in L^2(\partial_+ SM),\ u\circ \SS_A = \pm u\},
\end{align*}
where the antipodal scattering relation $\SS_A: \partial_+ SM\to \partial_+ SM$ is defined as
\begin{align*}
    \SS_A(\beta,\alpha) = (\beta+ \pi + 2\ss(\alpha), -\alpha),
\end{align*}
and for the purposes of understanding $I_0$, one may forget $L_-^2$. A Hilbert basis for $L^2_+$ that is adapted to the X-ray transform is $\{\psi_{n,k},\ n\ge 0,\ k\in \Zm\}$ as defined in \eqref{eq:psink}, whose $C^\infty$ span (i.e. expansions with rapid decay) generates the space $C_{\alpha,-,+}^\infty(\partial_+ SM)$ defined in \eqref{eq:Calmp}, as explained in \cite[Proposition 6]{Mishra2019}. Since the operator $-T^2 = -\sqrt{\ss'} \left( \partial_\beta - \frac{1}{\ss'(\alpha)} \partial_\alpha \right)^2 \frac{1}{\sqrt{\ss'}}$ is formally self-adjoint on $L^2_+$ with spectral decomposition as in \eqref{eq:psink2}, we may then define a functional calculus, namely we may define 
\begin{align*}
    f(-T^2) w &:= \sum_{n,k} f( (n+1)^2) w_{n,k} \widehat{\psi_{n,k}}, \qquad w = \sum_{n,k} w_{n,k}\ \widehat{\psi_{n,k}} \in C_{\alpha,-,+}^\infty(\partial_+ SM). 
\end{align*}
\begin{remark}
    Note that $(-T^2)^{1/2}$ is quite different from $T$, as $(-T^2)^{1/2}$ maps $\ker (Id-\SS_A^*)$ into itself, while $T$ maps $\ker(Id -\SS_A^*)$ into $\ker(Id + \SS_A^*)$.
\end{remark}

We can then define Sobolev scales associated with $-T^2$ as follows 
\begin{align*}
    H_{T,+}^s (\partial_+ SM) &= \{ w \in L_+^2 (\partial_+ SM, d\Sigma^2), \qquad (-T^2)^{s/2} w \in L_+^2(\partial_+ SM, d\Sigma^2 ) \} \\
    &= \{ w = \sum_{n,k} w_{n,k} \widehat{\psi_{n,k}}, \qquad \sum_{n,k} (n+1)^s |w_{n,k}|^2 <\infty \}.
\end{align*}

We now move to the proof of Theorem \ref{thm:main2}.

\begin{proof}[Proof of Theorem \ref{thm:main2}] 
    As mentioned in the introduction, there are two key things to prove: (i) the description of the cokernel of $I_0$, and (ii) the smoothing properties of $I_0$. 

    Regarding (i), recall the operator $C_-\colon L^2_+ \to L^2_+$ defined by $C_- := \frac{1}{2} A_-^* H_- A_-$ and introduced in \cite{Monard2015a,Mishra2019}. In \cite{Mishra2019}, it is shown that when $R= 1$ and $|\kappa|<1$, $C_-$ acts diagonally on the $\psi_{n,k}$ basis as follows\footnote{This is initially formulated in the $u'_{p,q}$ basis, followed by a reindexing $(n,k) \mapsto (p,q) = (n-2k,n-k)$ and defining $\psi_{n,k} = \frac{(-1)^n}{4\pi} u'_{n-2k,n-k}$}
    \begin{align*}
	C_- \psi_{n,k} = \frac{i}{2} (\sgn{2(n-k)+1} + \sgn{-(2k+1)}) \psi_{n,k} =  \left\{
	    \begin{array}{cl}
		i \psi_{n,k} & n\ge 0,\quad k<0, \\
		0 & n\ge 0, \quad 0\le k\le n, \\
		- i \psi_{n,k} & n\ge 0,\quad k>n.
	    \end{array}
	\right.
    \end{align*}
    The generalization to general $R$ and $\kappa$ such that $R^2 |\kappa|<1$ is identical and we do not repeat it here. From their diagonal action on $\psi_{n,k}$, the operators $C_-$ and $-T^2$ commute. Moreover, $C_-$ vanishes exactly on the range of $I_0$ while having spectral values in $\{\pm i\}$ on the orthocomplement. In other words, for every $s$, the operator 
    \begin{align*}
	C_-: H_{T,+}^s (\partial_+ SM, d\Sigma^2) \to H_{T,+}^s (\partial_+ SM, d\Sigma^2),
    \end{align*}
    is bounded, skew-adjoint, with operator norm $1$. 

    On to looking at (ii), we will quantify precisely the gain induced by $I_0$ on the Sobolev scales we have defined. In the reference case, Theorem \ref{thm:SVD} implies that 
    \begin{align*}
	I_0^e Z_{n,k}^e = \frac{4\pi}{n+1} \psi_{n,k}^e, \qquad n\ge 0, \qquad 0\le k\le n.
    \end{align*}
    We now plug in \eqref{eq:intertw_adj} and the definition of $Z_{n,k}$ \eqref{eq:Znk2} and $\psi_{n,k}$ \eqref{eq:psink2} into the previous equation to make appear
    \begin{align*}
	I_0 w (Z_{n,k}) = \frac{R (1+\kappa R^2)^{1/2}}{(1-\kappa R^2)^{3/2}} \frac{4\pi}{n+1} \psi_{n,k}, \qquad n\ge 0, \qquad 0\le k\le n,
    \end{align*}
    which in turn becomes 
    \begin{align*}
	I_0 w (\widehat{Z_{n,k}}) &= \frac{\|\psi_{n,k}\|}{\|Z_{n,k}\|} \frac{R (1+\kappa R^2)^{1/2}}{(1-\kappa R^2)^{3/2}} \frac{4\pi}{n+1} \widehat{\psi_{n,k}} = \frac{\sqrt{c_{\kappa,R}}}{\sqrt{n+1}} \widehat{\psi_{n,k}}, 
    \end{align*}
    upon using \eqref{eq:Znknorm2} and \eqref{eq:psinknorm2}. In particular, $I_0$ acts by gaining precisely a $\sqrt{n+1}$ decay from the $\wtH^\bullet$ scale to the $H_{T,+}^\bullet$ scale, and \eqref{eq:sharp} follows. Theorem \ref{thm:main2} is proved.     
\end{proof}

\appendix

\section{Appendix} \label{sec:appendix}

\subsection{Zernike facts}\label{sec:Zernike}

Let us explain how differentiation acts on the Zernike basis. Here we use the convention in \cite{Kazantsev2004}. Specifically, \cite[Theorem 1]{Kazantsev2004} states that 
\begin{align}
    Z_{n,k}(z,\zbar) = \frac{1}{k!} \frac{\partial^k}{\partial z^k} \left[ z^n \left( \frac{1}{z}-\zbar \right)^k \right], \qquad n\ge 0, \qquad 0\le k\le n.
    \label{eq:Zernike}
\end{align}

\begin{lemma} \label{lem:differentiation}
    The following properties holds: 
    \begin{align}
	\partial_z Z_{n,k} &= n Z_{n-1,k} - \partial_z Z_{n-2,k-1}, \label{eq:ppty4} \\
	\partial_{\zbar} Z_{n,k} &= - n Z_{n-1,k-1} - \partial_{\zbar} Z_{n-2,k-1}. \label{eq:ppty5}
    \end{align}
\end{lemma}

\begin{proof} To prove \eqref{eq:ppty4}, we write
    \begin{align*}
	\partial_z Z_{n,k} &= \frac{1}{k!} \partial_z^{k+1} \left( z^n (z^{-1}-\zbar)^k \right) \\
	&= \frac{1}{k!} \partial_z^k \left( nz^{n-1} (z^{-1}-\zbar)^k \right) - \frac{1}{(k-1)!} \partial_z^k \left( z^{n-2} (z^{-1}-\zbar)^{k-1} \right) \\
	&= n Z_{n-1,k} - \partial_z Z_{n-2,k-1}.
    \end{align*}
    Then the proof of \eqref{eq:ppty5} follows via
    \begin{align*}
	\partial_{\zbar} Z_{n,k} = - \partial_z Z_{n,k-1} &= - n Z_{n-1,k-1} + \partial_{z} Z_{n-2,k-2,} \\
	&= - n Z_{n-1,k-1} - \partial_{\zbar} Z_{n-2,k-1}.
    \end{align*}    
\end{proof}

Applying \eqref{eq:ppty4} until going out of bounds, we obtain
\begin{align}
    \partial_z Z_{n,k} = \sum_{p = 0}^{P_{n,k}} (n-2p) (-1)^p Z_{n-1-2p,k-p},
    \label{eq:diffZnk}
\end{align}
where
\begin{align}
    P_{n,k} := \left\{
	\begin{array}{cc}
	    k & \text{if } k< n-k, \\
	    n-k-1 & \text{if } k\ge n-k.
	\end{array}
	\right. \qquad n>0. \qquad P_{0,0} = 0.
    \label{eq:Pnk}
\end{align}
Similarly, for $k\ge 1$
\begin{align*}
    -\partial_{\zbar} Z_{n,k} = \partial_z Z_{n,k-1} = \sum_{p= 0}^{P_{n,k-1}} (n-2p) (-1)^{p} Z_{n-1-2p,k-1-p}.
\end{align*}

\paragraph{Zernike expansions of $\partial f$ and $\dbar f$.} Turning \eqref{eq:diffZnk} around, given $Z_{n,k}$, the only basis elements $Z_{n',k'}$ such that $(\partial_z Z_{n',k'}, Z_{n,k}) \ne 0$ are $Z_{n+1+2p,k+p}$ and such that 
\begin{align*}
    \partial_z Z_{n+1+2p, k+p} = \dots + (-1)^p (n+1) Z_{n,k} + \dots, \qquad p\ge 0.
\end{align*}
In particular, this means
\begin{align*}
    (\partial_z Z_{n+1+2p, k+p}, Z_{n,k}) = (-1)^p\ \pi, \qquad p\ge 0. 
\end{align*}
Any function $f\in L^2(\Dm)$ can be written as
\begin{align}
    f = \sum_{n=0}^\infty \frac{n+1}{\pi} \sum_{k=0}^n (f,Z_{n,k}) Z_{n,k}.
    \label{eq:fL2}
\end{align}
Then 
\begin{align*}
    \partial_z f = \sum_{n=0}^\infty \frac{n+1}{\pi} \sum_{k=0}^n (\partial_z f,Z_{n,k}) Z_{n,k},
\end{align*}
where 
\begin{align*}
    (\partial_z f, Z_{n,k}) &= \sum_{n'= 0}^\infty \sum_{k'=0}^{n'} \frac{n'+1}{\pi} (f, Z_{n',k'}) (\partial_z Z_{n',k'}, Z_{n,k})  \\
    &= \sum_{p=0}^\infty (n+2+2p) (-1)^p (f, Z_{n+1+2p, k+p}).
\end{align*}
Similar considerations for $\dbar$ yield that 
\begin{align*}
    (\dbar Z_{n+1+2p,k+1+p}, Z_{n,k}) = (-1)^{p+1} \pi, \qquad p\ge 0,
\end{align*}
and this implies the decomposition
\begin{align*}
    \dbar f = \sum_{n=0} \frac{n+1}{\pi} \sum_{k=0}^n (\dbar f, Z_{n,k}) Z_{n,k}, 
\end{align*}
where 
\begin{align*}
    (\dbar f, Z_{n,k}) &= \sum_{n',k'} \frac{n'+1}{\pi} (f, Z_{n',k'}) (\dbar Z_{n',k'}, Z_{n,k}) \\ 
    &= \sum_{p=0}^\infty (-1)^{p+1} (n+2+2p) (f, Z_{n+1+2p, k+1+p}). 
\end{align*}

\subsection{Proofs of missing lemmas}\label{sec:missing}

\begin{proof}[Proof of Lemma \ref{lem:continuous}] Since our definition of $Z_{n,k}$ agrees with \cite{Kazantsev2004}, we have the representation (see \cite[Eq. (4.2)]{Kazantsev2004})
    \begin{align*}
	Z_{n,k}(\rho e^{i\omega}) = (-1)^k e^{i(n-2k)\omega} \rho^{n-2k} P_k^{(0,|n-2k|)}(2\rho^2-1),
    \end{align*} 
    where $P_k^{(a,b)}$ refers to Jacobi polynomials. From \cite[Theorem 7.2 p. 163]{Szegoe1938}, we deduce that $\sup_{\Dm} |Z_{n,k}| = |Z_{n,k}(1)| = 1$. Combining this with \eqref{eq:Znknorm}, we obtain $\sup_\Dm |\widehat{Z_{n,k}}| = \frac{1}{\sqrt{\pi}} (n+1)^{1/2}$, and thus 
	\begin{align*}
	    \sum_{n,k} |f_{n,k}| |\widehat{Z_{n,k}}| &\le \frac{1}{\sqrt{\pi}} \sum_{n,k} |f_{n,k}| (n+1)^{1/2} \\
	    & \le \frac{1}{\sqrt{\pi}} \sum_{n,k} (n+1)^\alpha |f_{n,k}| (n+1)^{\frac{1}{2}-\alpha} \\
	    & \le \frac{1}{\sqrt{\pi}} \left( \sum_{n,k} (n+1)^{2\alpha} |f_{n,k}|^2 \right) \sum_{n,k} (n+1)^{1-2\alpha} \\
	    & \le \frac{1}{\sqrt{\pi}} \left( \sum_{n,k} (n+1)^{2\alpha} |f_{n,k}|^2 \right) \sum_{n} (n+1)^{2-2\alpha},
	\end{align*}
	where the last sum is finite whenever $\alpha > 3/2$. 
    \end{proof}

        \begin{proof}[Proof of Lemma \ref{lem:tame}]
	    We prove the statement for $\partial$, the estimates for $\dbar$ are similar. Let us recall the equation
	\begin{align*}
	    (\partial f, Z_{n,k}) = \sum_{p=0}^\infty (n+2+2p) (-1)^p (f, Z_{n+1+2p, k+p}). 
	\end{align*}
	Translating into normalized Zernike, this implies the relation
	\begin{align*}
	    (\partial f)_{n,k} = \sqrt{n+1} \sum_{p=0}^\infty (-1)^p (n+2+2p)^{\frac{1}{2}} f_{n+1+2p,k+p}. 
	\end{align*}
	In particular, we write
	\begin{align*}
	    |(\partial f)_{n,k}|^2 &\le (n+1) \left( \sum_{p=0}^\infty (n+2+2p)^{1/2} |f_{n+1+2p,k+p}| \right)^2 \\
	    &\le (n+1) \left(\sum_{p=0}^\infty (n+2+2p)^{1-2\beta}\right) \sum_{p=0}^\infty (n+2+2p)^{2\beta} |f_{n+1+2p,k+p}|^2 \\
	    &\le (n+1) \zeta(2\beta-1, n+1) \sum_{p=0}^\infty (n+2+2p)^{2\beta} |f_{n+1+2p,k+p}|^2,
	\end{align*}
	where $\zeta(s,q) := \sum_{p=0}^\infty (q+p)^{-s}$ is the Hurwitz Zeta function, convergent for $s>1$ so the estimate above holds for all $\beta >1$. Moreover, with the obvious crude estimate $\zeta(s,q) \le \zeta(s)$, we write the first estimate 
	\begin{align*}
	    |(\partial f)_{n,k}|^2 &\le \zeta(2\beta-1) (n+1) \sum_{p=0}^\infty (n+2+2p)^{2\beta} |f_{n+1+2p,k+p}|^2
	\end{align*}
	We then compute
	\begin{align*}
	    \|\partial f\|^2_{\wtH^{\alpha-\ell}} &= \sum_{n,k} (n+1)^{2(\alpha-\ell)} |(\partial f)_{n,k}|^2 \\
	    &\le \zeta(2\beta-1) \sum_{n,k,p} (n+1)^{2(\alpha-\ell)+1} (n+2+2p)^{2\beta} |f_{n+1+2p,k+p}|^2 \\
	    &\le \zeta(2\beta-1) \sum_{n',k'} |f_{n',k'}|^2 \sum_{n,k,p} (n+1)^{2(\alpha-\ell)+1} (n+2+2p)^{2\beta}
	\end{align*}
	where the latter sum holds over the $n\ge 0$, $0\le k\le n$ and $p\ge 0$ such that $n+1+2p = n'$ and $k+p=k'$. At fixed $n',k'$, given $p\ge 0$, $n,k$ are determined. Moreover the two constraints impose $0\le p\le P_{n',k'}$ as defined in \eqref{eq:Pnk}. We thus arrive at
	\begin{align*}
	    \|\partial f\|^2_{\wtH^{\alpha-\ell}} &\le \zeta(2\beta-1)  \sum_{n',k'} |f_{n',k'}|^2 \sum_{p=0}^{P_{n',k'}} (n'-2p)^{2(\alpha-\ell)+1} (n'+1)^{2\beta} \\
	    &\le \zeta(2\beta-1) \sum_{n',k'} |f_{n',k'}|^2 (n'+1)^{2(\alpha-\ell) + 2 + 2\beta},  
	\end{align*}  
	upon bounding crudely $\sum_{p=0}^{P_{n',k'}} (n'-2p)^{2(\alpha-\ell)+1} \le (n'+1)^{2(\alpha-\ell)+2}$. The last right-hand side is then controlled by $\|f\|^2_{\wtH^\alpha}$ if we choose $\ell = \beta + 1$. Since $\beta$ can be chosen as $1+\varepsilon$ for any $\varepsilon>0$, the result follows. 
    \end{proof}

\subsection{A functional-analytic proof of \eqref{eq:Znknorm}}

The space $H^1_0 (\Dm) \F{=H_0^1(\Dm,\Cm)}$ can be endowed with three equivalent norms
\begin{align*}
    \|\partial_x u\|_{\Dm}^2 + \|\partial_y u\|_{\Dm}^2 = 4 \|\partial u\|_\Dm^2 = 4 \|\dbar u\|_\Dm^2, \qquad u\in H^1_0(\Dm),  
\end{align*} 
\F{where we denote $(u,v)_{\Dm} := \int_{\Dm} u \bar{v}$ and $\|u\|_{\Dm}^2 := (u,u)_{\Dm}$. } Using Riesz representation on the second norm, any linear form on $H^1_0(\Dm)$ can be uniquely written as $v\mapsto (\partial f, \partial v)_\Dm$ for some $f\in H^1_0(\Dm)$, or upon setting $u = \partial f\in (L^2(\ker \dbar))^\perp$, any linear form on $H^1_0(\Dm)$ can be uniquely written as $v\mapsto (u, \partial v)_\Dm$ for some $u\in (L^2(\ker \dbar))^\perp$. Now given $u\in L^2(\Dm)$, the mapping $v\mapsto -(u, \dbar v)_\Dm$ is a linear form on $H^1_0(\Dm)$ and as such, there exists a unique $Bu \in (L^2(\ker \dbar))^\perp$ such that
\begin{align}
    (Bu, \partial v) = - (u, \dbar v), \qquad \forall\ v\in H^1_0(\Dm),
    \label{eq:propB}
\end{align}
with the estimate $\|Bu\|_\Dm \le \|u\|_\Dm$. We call $B$ the\footnote{or rather, a realization of the Beurling transform. For bounded domains, other boundary conditions can be prescribed in, e.g. \cite{Astala2009}. Other Beurling transforms defined in terms of ladder operators on fiber-harmonic decompositions on the unit tangent bundle appear in, e.g., \cite{Paternain2015}.} {\bf Beurling transform}, $B\colon L^2(\Dm) \to L^2(\Dm)$ with norm at most $1$. If $u$ is smooth enough, then $Bu$ is such that $-\dbar (Bu) = \partial u$. Now with the property that $\dbar Z_{n,k+1} = -\partial Z_{n,k}$ and the fact that $Z_{n,k+1} \perp \langle Z_{p,0},\ p\ge 0\rangle$, this precisely means that $Z_{n,k+1} = BZ_{n,k}$ for every $0\le k\le n-1$. 

\begin{proof}[Proof of \eqref{eq:Znknorm}] Since $Z_{n,0} = z^n$ and $Z_{n,n} = (-1)^n \zbar^n$, the proof that 
    \begin{align*}
	\|Z_{n,0}\|^2 = \|Z_{n,n}\|^2 = \frac{\pi}{n+1}
    \end{align*}
    is a straightforward computation. In addition, since the Beurling transform has norm not exceeding $1$, and with $Z_{n,k} = B^k Z_{n,0}$ for all $0\le k\le n$, we deduce that 
    \begin{align*}
	\frac{\pi}{n+1} = \|Z_{n,n}\|^2 \le \|Z_{n,n-1}\|^2 \le \dots\le \|Z_{n,1}\|^2 \le \|Z_{n,0}\|^2 = \frac{\pi}{n+1},
    \end{align*}
    hence all these norms equal $\frac{\pi}{n+1}$.     
\end{proof}

\section*{Acknowledgements}

\F{The author acknowledges fruitful discussions with Bill Lionheart, who was separately aware of the existence of the operator $\DD$ in \eqref{eq:DD} at the time of redaction of this article. Helpful comments from Gabriel Paternain, Todd Quinto and the anonymous referees are also gratefully acknowledged.}

This material is based upon work supported by the National Science Foundation under grants DMS-1814104 and DMS-1440140, while the author was in residence at the Mathematical Sciences Research Institute in Berkeley, California, during the Fall 2019 Semester.

\bibliographystyle{siam}
\bibliography{../../../000-TexTouch/bibliography/bibliography}

\end{document}